\documentclass[a4paper,oneside,10.5pt]{article}%
\usepackage{amsmath}
\usepackage{amsfonts}
\usepackage{indentfirst}
\usepackage{amssymb}
\usepackage{graphicx}%
\usepackage{appendix}

\usepackage[colorlinks,citecolor=blue,urlcolor=blue]{hyperref}

\providecommand{\U}[1]{\protect \rule{.1in}{.1in}}

\newtheorem{theorem}{Theorem}[section]

\newtheorem{corollary}[theorem]{Corollary}

\newtheorem{definition}[theorem]{Definition}
\newtheorem{assumption}[theorem]{Assumption}

\newtheorem{lemma}[theorem]{Lemma}

\newtheorem{proposition}[theorem]{Proposition}
\newtheorem{remark}[theorem]{Remark}

\newenvironment{proof}[1][Proof]{\noindent \textbf{#1.} }{\  \rule{0.5em}{0.5em}}
\numberwithin{equation}{section}

\begin{document}
	
	\title{Stochastic maximum principle for optimal control of infinitely delayed systems of functional type in infinite dimensions}
	\author{Guanwei Cheng\thanks{School of Mathematics, Shandong University, PR China, (guanwei.cheng@mail.sdu.edu.cn).}
        }
	\date{May 2026}
	\maketitle
	
	\textbf{Abstract}:  
This paper studies the optimal control problems of stochastic evolution equations with infinite delay of general functional type. By introducing a non-anticipative path derivative and its infinite-window dual operator, we derive the infinitely anticipated backward stochastic evolution equation (IABSEE) as the adjoint equation, and establish both necessary and sufficient maximum principles. As an application, we solve an infinite-delay linear-quadratic (LQ) control problem and obtain the explicit optimal control.
		
		\addcontentsline{toc}{section}{\hspace*{1.8em}Abstract}

\section{Introduction.}

Stochastic delayed systems have emerged as a fundamental framework for modeling complex phenomena where the evolution of the state depends not only on the current status but also on the historical trajectory.
In real-world applications, delayed systems frequently exhibit a long-term memory effect, where the impact of past events decay slowly but persist over an infinite horizon.
Such effect is prevalent in diverse fields, such as financial engineering (e.g., long-range dependence in volatility \cite{Volatility-long-term} and asset returns \cite{return-long-term}), biological systems (e.g., cumulative hereditary factors in population growth \cite{population-long-term}), and climate modeling (e.g., historical $CO_2$ concentration dynamics \cite{carbon}). 
Consequently, establishing a rigorous control theory for infinitely delayed stochastic systems is imperative to accurately characterize these persistent memory effects.

Over the past decades, the Stochastic Maximum Principle (SMP) has been widely developed for systems with various types of finite delays. 
For point-wise delays, where history is captured at discrete lags, significant progress has been achieved in both finite dimensions (see e.g. the pioneering work of Chen and Wu \cite{SDDE1wuzhenchenli}) and infinite dimensions (see e.g. Dai, Zhou and Li \cite{SDDE-IH-1delay}).
Systems with moving-average delays, which integrate past trajectories over a finite window, have also been investigated in diverse settings. For instance, one may refer to Øksendal, Sulem and Zhang \cite{SDDE-12delay} for the mixture of point-wise and moving-average delays, to Agram and Øksendal \cite{SDDE-12delay+IH} for the infinite-horizon case.
Furthermore, optimal control for systems with integral delays relative to general finite measures has been explored in Chen and Huang \cite{BSDDE-12.delay-1type-1}, Guatteri and Orrieri \cite{SDDE-12.delay-2type-2}.
Notably, Liu, Song and Wang \cite{songjian2025} recently pushed the boundaries of this field by establishing SMP for general functional dependence (path-dependent) systems on Hilbert space, a domain that remains remarkably sparse.

Despite these advancements, existing works are fundamentally constrained by finite delay windows. To the best of our knowledge, the only extension to the infinite-delay setting was recently provided by Cheng \cite{ziji2} for finite-dimensional case, where the delay was restricted to an integral form. These gaps motivate us to bridge two frontiers: extending the general functional dependence to an infinite-delay structure in Hilbert spaces, thereby providing a unified framework for long-memory stochatic systems.

In this paper, we consider a class of optimal  stochastic control problem of infinitely delayed stochastic evolution equation (SEE) in the Hilbert space $H$. Let $\{W(t)\}_{t \geqslant 0}$ be a cylindrical Brownian motion. The state process $X(\cdot)$ is governed by
		\begin{equation}\label{state eq in introduction}
			\left\{
			\begin{aligned}
				dX(t) &= b(t,X_t, v_d(t))dt + \sigma(t,X_t,v_d(t))dW(t), \qquad t\in[0,T],\\
				X(t) &= \gamma(t), \quad v(t) = \varphi(t), \qquad t\in (-\infty, 0],
			\end{aligned}
			\right.
		\end{equation}
where $X_t := \{X(t-\theta): 0 \leqslant \theta < +\infty\}$ represents the general functional dependence of infinite path information, $v_d(t) := \int_{0}^{+\infty}\phi(t-\theta,t)v(t-\theta)\alpha(d\theta)$ means that the control is delayed via an integral with respect to a general finite measure $\alpha$. Our objective is to minimize a cost functional involving both running and terminal costs
	\begin{equation}\label{cost functional in introduction}
	\begin{aligned}
		J(v(\cdot))=\mathbb{E} \left[\int_{0}^{T} l\left(t, X_t, v_d(t)\right) d t\right]+ \mathbb{E}\big[h(X(T))\big].
	\end{aligned}
\end{equation}

The formulation in (\ref{state eq in introduction})-(\ref{cost functional in introduction}) represents a stochastic control system with infinite delay of general functional type, which covers above-mentioned systems with point-wise, moving-average and integral delays.
The main contributions of this paper are twofold, which also represents two key components in deriving the SMP for our optimal control problem:

\begin{itemize}
    \item \textbf{Fading-memory path space in infinite-dimensions.}     
    As mentioned in \cite{choice-phasespace2, choice-phasespace1}, 
    the choice of an appropriate phase space is the primary challenge when extending functional systems to the infinite-delay setting. 
    In contrast to existing literature on infinitely-delayed functional differential equations that primarily focuses on the well-posedness in finite dimensions (see e.g. \cite{07JMAA, 17JDE}), our work employs a robust functional framework on the infinite-dimensional fading-memory path space $C_\lambda((-\infty, T]; H)$ (see Definition~\ref{def: fading-memory path space}). 
    This space not only serves as the phase space for the well-posedness of the infinitely delayed SEE (see Theorem~\ref{thm:IABSDE-wellposed}),
    but also enables the rigorous analysis of path derivatives and their infinite-window dual operators (see Section~\ref{chapter Non-anticipative path derivatives and infinite-window dual operators in Hilbert spaces}), which forms the duality relationships central to SMP.

    \item \textbf{Dual analysis on infinite horizons and derivation of adjoint equation.} 
    We perform a dual analysis on the product space of infinite time and state, extending the finite-window work of \cite{songjian2025} to an unbounded horizon. 
    Specifically, we introduce the non-anticipative path derivative on $C_\lambda((-\infty, T]; H)$ (see Definition~\ref{def: non-anticipatively differentiable}), and construct its Hilbert adjoint (see Theorem~\ref{thm:R-adjoint}). 
    By proving the strong convergence and boundedness of this adjoint operator, we characterize its adapted version via an optional projection that integrates whole future path segments over $[0, \infty)$ (see \eqref{eq:adapted-adjoint}). This construction overcomes the finite-window constraints in \cite{songjian2025}, enabling the natural formulation of the infinitely anticipated backward stochastic evolution equation (IABSEE) as the adjoint equation. Furthermore, the well-posedness of the resulting IABSEE is established (see Theorem~\ref{thm:IABSDE-wellposed}).

\end{itemize}

The remainder of this paper is structured as follows. 
In Section~\ref{chapter preliminary}, we introduce some notations and address the well-posedness of both the infinitely delayed SEE and the IABSEE.
Section~\ref{chapter Non-anticipative path derivatives and infinite-window dual operators in Hilbert spaces} is dedicated to a detailed investigation of non-anticipative path derivatives and their associated dual operators.
Building on the duality developments, Section~\ref{section SMP} derives the necessary and sufficient conditions for maximum principles.
Finally, the theoretical results are illustrated through a linear-quadratic (LQ) control problem in Section~\ref{chapter applications}.

\section{Preliminary.}\label{chapter preliminary}
\subsection{Notations and Functional Spaces.}\label{section Notations and Functional Spaces}
        For generic Banach spaces $X$ and $Y$, $\mathcal{L}(X, Y)$ denotes the space of bounded linear operators mapping from $X$ to $Y$ with norm  $\|\cdot\|_{\mathcal{L}(K,H)}$.
		Let $H$ and $K$ be real separable Hilbert spaces with scalar products $\left\langle \cdot, \cdot \right\rangle_{H}$, $\left\langle \cdot, \cdot \right\rangle_{K}$ and norms $\|\cdot\|_H$, $\|\cdot\|_K$ respectively, we denote by $\mathcal{L}_2(K, H)$ the space of Hilbert-Schmidt operators from $K$ to $H$, which is a separable Hilbert space  endowed with the norm $\|\cdot\|_{\mathcal{L}_2}$ and the inner product 
		\begin{equation}\nonumber
			\left\langle A, B \right\rangle_{2} := \mathrm{Tr}\left(B^{*}A\right), \quad A, B \in \mathcal{L}_2(K,H),
		\end{equation}
        where $B^*\in \mathcal{L}_2(H,K)$ denotes the adjoint operator of $B$.

Let $(\Omega,\mathcal F,\{\mathcal F_t\}_{t\geqslant0}, P)$ be a complete filtered probability space
satisfying the usual conditions, and let
$W=\{W(t)\}_{t\geqslant0}$ be an $m$-dimensional cylindrical Brownian motion on $K = \mathbb{R}^m$,
adapted to $\{\mathcal F_t\}_{t\geqslant0}$.

For a sub-$\sigma$-algebra $\mathcal G\subset\mathcal F$, we define
\begin{itemize}
    \item $L^2(\mathcal G;H)$: the space of $\mathcal G$-measurable $H$-valued random variables $\xi$
    such that
    \[
    \mathbb E\big[\|\xi\|_H^2\big]<\infty.
    \]
\end{itemize}

For an interval $I\subset\mathbb R$, we define
\begin{itemize}
 \item $L^2(I;H)$: the space of measurable $H$-valued functions $x(\cdot)$ such that
    \[
    \int_{I} \|x(s)\|_H^2\,ds<\infty.
    \]
    
    \item $\mathcal M_{\mathcal F}^2(I;H)$: the space of progressively measurable $H$-valued processes
    $X(\cdot)$ such that
    \[
    \mathbb E\left[\int_I \|X(t)\|_H^2\,dt\right]<\infty.
    \]

    \item $\mathcal M_{\mathcal F}^{2,\beta}(I;H)$: for $\beta>0$, the exponentially weighted space of
    progressively measurable $H$-valued processes $X(\cdot)$ such that
    \[
    \mathbb E\left[\int_I e^{\beta t}\|X(t)\|_H^2\,dt\right]<\infty.
    \]

    \item $\mathcal S_{\mathcal F}^2(I;H)$: the space of continuous, $\mathcal F$-adapted $H$-valued
    processes $X(\cdot)$ such that
    \[
    \mathbb E\left[\sup_{t\in I}\|X(t)\|_H^2\right]<\infty.
    \]
\end{itemize}

For a continuous interval $I$ with endpoints $-\infty \leqslant a < b \leqslant +\infty$, our definitions above cover various types of interval, exemplified by $L^2(-\infty, 0; H)$ with $I = (-\infty, 0]$, $\mathcal{S}_{\mathcal{F}}^2(0, T; H)$ with $I = [0, T]$, and $\mathcal{M}_{\mathcal{F}}^{2, \beta}(T, \infty; H)$ with $I = [T, \infty)$. 
\\
\\
\indent
Next, we introduce the fading-memory path space $C_\lambda((-\infty, T]; H)$ to analyze infinite-delay effects. 
\begin{definition}\label{def: fading-memory path space}
    Fixing $\lambda>0$ and $T>0$, we define 
		\[
		C_\lambda((-\infty,T];H)
		:=\Big\{x\in C((-\infty,T];H):
		\|x\|_{\lambda,T}:=\sup_{\theta\leqslant T}e^{\lambda \theta}\|x(\theta)\|_H<\infty,\ 
		\lim_{\theta\to-\infty}e^{\lambda\theta}x(\theta)=0\Big\},
		\]
		where $C((-\infty,T];H)$ denotes the family of continuous functions from $(-\infty,T]$ to $H$. This space will be used as the phase space of processes with infinite delay. We refer to Section~\ref{section 5.1} for more detailed properties and the associated non-anticipative path operators acting on this phase space.
\end{definition}

\medskip
Moreover, we introduce the It\^o's formula on Hilbert space which will be used in the following chapters.

\begin{lemma}[It\^o's formula on Hilbert space]
Consider an $H$-valued stochastic process $\{X(t)\}_{t \in [0,T]}$ satisfying the following stochastic differential equation:
\begin{equation*}
X(t) = X(0) + \int_0^t b(s) ds + \int_0^t \sigma(s) dW(s), \quad t \in [0,T],
\end{equation*}
where $b \in \mathcal{M}_{\mathcal{F}}^2(0,T; H)$ and $\sigma \in \mathcal{M}_{\mathcal{F}}^2(0,T; \mathcal{L}_2(\mathbb{R}^m,H))$. Let $\pi: [0,T] \times H \to \mathbb{R}$ be a function of class $C^{1,2}$, that is, $\pi$ is continuously differentiable in time $t$ and Fr\'echet differentiable of order 1 and 2 with respect to $x\in H$ with first and second order Fr\'echet derivatives
\[
D\pi(t,x) \in H, \quad D^2\pi(t,x)\in \mathcal{L}(H,H).
\]
Then, for all $t \in [0,T]$, the following formula holds $P$-a.s.
\begin{equation}\label{ito on scalar function}
\begin{aligned}
\pi(t, X(t)) = & \pi(0, X(0)) + \int_0^t \left[ \frac{\partial \pi}{\partial s}(s, X(s)) + \langle D \pi(s, X(s)), b(s) \rangle_H \right. \\
& \left. + \frac{1}{2} \mathrm{Tr}\left( \sigma(s) \sigma^*(s) D^2 \pi(s, X(s)) \right) \right] ds + \int_0^t \langle D\pi(s, X(s)), \sigma(s) dW(s) \rangle_H.
\end{aligned}
\end{equation}
Furthermore, let $Y(\cdot), Z(\cdot)$ be two $H$-valued It\^o processes with diffusion coefficients $\sigma_Y(t), \sigma_Z(t)$, respectively. Then, for all $t \in [0,T]$, the following product It\^o's formula holds $P$-a.s.
\begin{equation}\label{ito on 2 product}
d\langle Y(t), Z(t) \rangle_H = \langle dY(t), Z(t) \rangle_H + \langle Y(t), dZ(t) \rangle_H + \langle\sigma_Y(t) ,\sigma_Z(t)\rangle_2 dt.
\end{equation}
\end{lemma}

These infinite-dimensional It\^o's formulas are justified by the work of Curtain and Falb \cite{Itolemma}. Formula (\ref{ito on scalar function}) is the real-valued case of Corollary 3.34 in \cite{Itolemma}, and the product rule (\ref{ito on 2 product}) follows directly by applying their main Theorem 3.8 to the bilinear inner product functional.

		\subsection{Infinitely delayed stochastic evolution equations.}
		Throughout this paper, $T>0$ is fixed, and for any process $\Gamma(\cdot)$, we write $\Gamma(t)$ for its value at time $t$, and 
		\[
		\Gamma_t := \Big\{\Gamma(t-\theta), 0 \leqslant \theta < +\infty \Big\}, \quad 
		\Gamma_{t+} := \Big\{ \Gamma(t+\theta), 0 \leqslant \theta < +\infty \Big\}
		\]
		to denote its past and future trajectories, respectively. 
		We consider the following infinitely delayed stochastic evolution equation (SEE):
		\begin{equation}\label{ISDDE without control}
			\left\{
			\begin{aligned}
				dX(t) &= b(t,X_t)\,dt + \sigma(t,X_t)\,dW(t), \qquad t\in[0,T],\\
				X(t) &= \gamma(t), \quad t\in (-\infty, 0],
			\end{aligned}
			\right.
		\end{equation}
		where the initial data $X_0 = \gamma_0 = \left\{\gamma(\theta), -\infty < \theta \leqslant 0\right\} \in L^2(\mathcal{F}_0;C_\lambda((-\infty,0];H))$ such that $\gamma(\cdot)\in \mathcal{M}_\mathcal{F}^2(-\infty, 0 ; H)$, and 
		the coefficients
		\[
		b:[0,T]\times \Omega\times C_\lambda((-\infty,T];H)\to H,\qquad 
		\sigma:[0,T]\times \Omega\times C_\lambda((-\infty,T];H)\to \mathcal{L}_2(\mathbb{R}^m,H)
		\]
		are progressively measurable in $(t,\omega)$ for each phase path in $C_\lambda((-\infty,T];H)$.
	\begin{definition} \label{definition X}
				An $\mathcal{F}$-adapted continuous process $X(\cdot)$ on $(-\infty,T]$ is called a solution to the infinitely delayed SEE
		(\ref{ISDDE without control}) if (\ref{ISDDE without control}) holds $P$-a.s. and $X_t\in C_\lambda((-\infty,t];H)$ for all $t\in[0,T]$.
	\end{definition}

\begin{remark}
    An $H$-valued process $X(\cdot)$ defined on $(-\infty,T]$ has its trajectories $X_t$ in $C_\lambda((-\infty,t];H)$ whenever $X(\cdot)$ has continuous paths and $\sup_{s\leqslant t}e^{\lambda(s-t)}\|X(s)\|_H<\infty$. 
\end{remark}
		\medskip

		In order to attain the well-posedness of SEE (\ref{ISDDE without control}), we propose the following conditions:
		\begin{assumption}\label{ass:ISFDE-Lip}
			(Uniform Lipschitz condition)
			There exists a constant $L_1>0$ such that for all $t\in[0,T]$,
			$x, x^{\prime}\in C_\lambda((-\infty,T];H)$,
			\[
			\|b(t, x)-b(t, x^{\prime})\|_H
			+\|\sigma(t,x)-\sigma(t,x^{\prime})\|_{\mathcal{L}_2}
			\leqslant L_1 \|x-x^{\prime}\|_{\lambda,T}.
			\]
		\end{assumption}
		
		\begin{assumption}\label{ass:ISFDE-int}
			\[
			\mathbb{E}\left[\int_0^T \|b(t,0)\|_H^2 dt\right]
			+\mathbb{E}\left[\int_0^T \|\sigma(t,0)\|_{\mathcal{L}_2}^2dt\right] <\infty.
			\]
		\end{assumption}
		
We start by proving a prior estimate for the solution of SEE (\ref{ISDDE without control}). In this paper, we allow the positive constant $C$ to change from line to line.

\begin{theorem}\label{thm:stability}
	For $i=1,2$, let $(b_{i},\sigma_{i})$ be two sets of coefficients such that Assumptions \ref{ass:ISFDE-Lip}--\ref{ass:ISFDE-int} hold, $\gamma_0^{(i)}\in L^2(\mathcal{F}_0;C_\lambda((-\infty,0];H)) $ be the initial datum. Let  $X^{(i)}(\cdot)$ be the corresponding solutions to SEEs (\ref{ISDDE without control}). Then there exists a constant $C > 0$, depending only on $T$ and $L_1$, such that  
	\begin{align}
		\mathbb{E}\Big[&\sup_{0\leqslant t\leqslant T}\|X^{(1)}(t)-X^{(2)}(t)\|_H^2\Big]
		\leqslant C \bigg\{ \mathbb{E}\Big[\|\gamma^{(1)}_0-\gamma^{(2)}_0\|_{\lambda,T}^2 \Big] \nonumber\\
		&\quad + \mathbb{E}\bigg[\int_0^T 
		\|b_1(t,X^{(2)}_t)-b_2(t,X^{(2)}_t)\|_H^2\,dt \bigg] 
		+ \mathbb{E}\bigg[\int_0^T 
		\|\sigma_1(t,X^{(2)}_t)-\sigma_2(t,X^{(2)}_t)\|_{\mathcal{L}_2}^2\,dt\bigg]\bigg\}. \label{eq:stability}
	\end{align}
\end{theorem}

\begin{proof}
	For $t\in(-\infty,T]$, denote $\hat{X}(t):=X^{(1)}(t)-X^{(2)}(t)$, which  satisfies
	\begin{equation*}
		\hat{X}(t)=\hat{X}(0)+\int_0^t \big(b_1(s,X^{(1)}_s)-b_2(s,X^{(2)}_s)\big)\,ds
		+\int_0^t \big(\sigma_1(s,X^{(1)}_s)-\sigma_2(s,X^{(2)}_s)\big)\,dW(s),
	\end{equation*}
	with initial data $\hat{X}(t)=\gamma^{(1)}(t)-\gamma^{(2)}(t)$ for $t\leqslant 0$.
	Applying the Hilbert-space It\^{o}'s formula (\ref{ito on scalar function}) to $\|\hat{X}(t)\|_H^2$ on $[0,T]$ yields
	\begin{align}
		\|\hat{X}(t)\|_H^2 - \|\hat{X}(0)\|_H^2 =&
		2\int_0^{t}\langle \hat{X}(s), b_1(s,X^{(1)}_s)-b_2(s,X^{(2)}_s)\rangle_H ds \nonumber \\
		&+2\int_0^{t}\langle \hat{X}(s),\big[\sigma_1(s,X^{(1)}_s)-\sigma_2(s,X^{(2)}_s)\big]dW(s)\rangle_H \nonumber\\		&+\int_0^{t}\|\sigma_1(s,X^{(1)}_s)-\sigma_2(s,X^{(2)}_s)\|_{\mathcal{L}_2}^2ds,  \label{eq:Ito-local}
	\end{align}
	Note that for all $s \in [0, T]$,
	\begin{align}
	\|X^{(1)}_s-X^{(2)}_s\|_{\lambda,T}
	\leqslant& \sup_{\theta\leqslant -s} e^{\lambda\theta}
	\|X^{(1)}(s+\theta)-X^{(2)}(s+\theta)\|_H\nonumber \\
    &+ \sup_{-s\leqslant\theta\leqslant 0} e^{\lambda\theta}
	\|X^{(1)}(s+\theta)-X^{(2)}(s+\theta)\|_H \nonumber \\
	\leqslant & \|\gamma^{(1)}_0-\gamma^{(2)}_0\|_{\lambda,T} + \sup_{0\leqslant r \leqslant s}\|\hat X(r)\|_H \nonumber,
	\end{align}
	Then by Assumption \ref{ass:ISFDE-Lip}, the Cauchy--Schwarz inequality and Young's inequality
	\begin{align}
		2\int_0^t& \langle \hat X(s), b_1(s,X^{(1)}_s)-b_2(s,X^{(2)}_s)\rangle_H ds \nonumber \\
		& \leqslant \int_0^t \left(\|\hat X(s)\|^2_H + \|b_1(s,X^{(1)}_s)-b_2(s,X^{(2)}_s)\|^2_H\right)ds \nonumber \\
		&\leqslant TL_{1}\|\gamma_0^{(1)}-\gamma_0^{(2)}\|_{\lambda,T}^2 + 
		(L_1 + 1)\int_0^t \sup_{0\leqslant r\leqslant s}\|\hat X(r)\|_H^2 ds \nonumber \\
		&\quad + L_1 \int_0^t
		\|b_1(s,X^{(2)}_s)-b_2(s,X^{(2)}_s)\|_H^2 ds . \label{eq:drift-est}
	\end{align}
	Similarly, we can derive
		\begin{align}
		\int_0^{t}\|\sigma_1(s,X^{(1)}_s)-\sigma_2(s,X^{(2)}_s)\|_{\mathcal{L}_2}^2ds 
		\leqslant& TL_{1}\|\gamma_0^{(1)}-\gamma_0^{(2)}\|_{\lambda,T}^2 + 
		L_1\int_0^t \sup_{0\leqslant r\leqslant s}\|\hat X(r)\|_H^2 ds \nonumber \\
		& + L_1 \int_0^t
		\|\sigma_1(s,X^{(2)}_s)-\sigma_2(s,X^{(2)}_s)\|_{\mathcal{L}_2}^2 ds . \label{eq:diffusion-est}
	\end{align}
	Moreover, applying the Burkholder--Davis--Gundy inequality, we obtain
	\begin{align}
	\mathbb{E}&\bigg[\sup_{0\leqslant t\leqslant T}
	\int_0^t\!\!\langle \hat X(s),
	[\sigma_1(s,X^{(1)}_s)-\sigma_2(s,X^{(2)}_s)]\,dW(s)\rangle_H\bigg] \nonumber\\
	& \leqslant C\mathbb{E}\Big(
	\int_0^T \|\hat X(s)\|_H^2
	\|\sigma_1(s,X^{(1)}_s)-\sigma_2(s,X^{(2)}_s)\|_{\mathcal{L}_2}^2 ds
	\Big)^{1/2} \nonumber \\
	&\leqslant \frac{1}{4}\mathbb{E}\bigg[\sup_{0\leqslant s\leqslant T}\|\hat X(s)\|_H^2\bigg]
	+ C\mathbb{E}\bigg[\int_0^T
	\|\sigma_1(s,X^{(1)}_s)-\sigma_2(s,X^{(2)}_s)\|_{\mathcal{L}_2}^2 ds \bigg].
	\label{eq:BDG}
	\end{align}
	Taking supremum over $t\in[0,\tau]$ for $\tau \in (0, T]$, taking expectation on both sides of \eqref{eq:Ito-local}, and combining \eqref{eq:drift-est}, \eqref{eq:diffusion-est}
	and \eqref{eq:BDG}, we obtain
	\begin{align*}
		\mathbb{E}\bigg[\sup_{0\leqslant t\leqslant \tau}\|\hat X(t)\|_H^2\bigg]\leqslant
		&C \mathbb{E}\bigg[\|\gamma_0^{(1)}-\gamma_0^{(2)}\|_{\lambda,T}^2\bigg]
		+ C\int_0^{\tau} \mathbb{E}\bigg[\sup_{0\leqslant r\leqslant s}\|\hat X(r)\|_H^2\bigg] ds  \\
		& +C\mathbb{E}\bigg[\int_0^{\tau}
		\|b_1(s,X^{(2)}_s)-b_2(s,X^{(2)}_s)\|_H^2 ds\bigg] \\
		& +C\mathbb{E}\bigg[\int_0^{\tau}
		\|\sigma_1(s,X^{(2)}_s)-\sigma_2(s,X^{(2)}_s)\|_{\mathcal{L}_2}^2 ds \bigg].
	\end{align*}
	By Gr\"onwall's inequality, the desired stability estimate (\ref{eq:stability}) follows.
\end{proof}

\begin{corollary}\label{cor:apriori}
	Suppose Assumptions \ref{ass:ISFDE-Lip}--\ref{ass:ISFDE-int} hold, let $X(\cdot)$ be the solution to SEE (\ref{ISDDE without control}) in the sense of Definition \ref{definition X}. Then for any given initial data $\gamma_0 \in L^2(\mathcal{F}_0;C_\lambda((-\infty,0];H))$ such that $\gamma(\cdot)\in \mathcal{M}_\mathcal{F}^2(-\infty, 0 ; H)$, there exists a constant $C>0$, depending only on $T$ and $L_1$, such that
	\begin{equation}\label{a priori est SEE}
	    	\mathbb{E}\Big[\sup_{-\infty < t\leqslant T}\|X(t)\|_H^2\Big]
	\leqslant C\left\{
	\mathbb{E}\Big[\|\gamma_0\|_{\lambda,T}^2\Big]
	+\mathbb{E}\Big[\int_0^T \|b(t,0)\|_H^2\,dt\Big]
	+\mathbb{E}\Big[\int_0^T \|\sigma(t,0)\|_{\mathcal{L}_2}^2\,dt\Big]
	\right\}.
	\end{equation}
	Moreover, we have $X(\cdot) \in \mathcal{S}_\mathcal{F}^2(-\infty, T ; H)$ and 
    \[
    \left\{b\left(t, X_t\right)\right\}_{t \in[0, T]} \in \mathcal{M}_{\mathcal{F}}^2(0, T ; H),\quad \left\{\sigma\left(t, X_t\right)\right\}_{t \in[0, T]}\in \mathcal{M}_{\mathcal{F}}^2\left(0, T ; \mathcal{L}_2(\mathbb{R}^m,H) \right).
    \]
\end{corollary}

	\medskip

Now, we give the well-posedness of SEE \eqref{ISDDE without control}.
		\begin{theorem}\label{thm:ISFDE-wellposed}
			Suppose Assumptions \ref{ass:ISFDE-Lip}-\ref{ass:ISFDE-int} hold.
			Then for any given initial data $\gamma_0 \in L^2(\mathcal{F}_0$; $C_\lambda((-\infty,0];H))$ such that $\gamma(\cdot)\in \mathcal{M}_\mathcal{F}^2(-\infty, 0 ; H)$, the SEE \eqref{ISDDE without control} admits a unique solution $X(\cdot)$ in the sense of Definition \ref{definition X}.
		\end{theorem}
		
		\begin{proof}
			The uniqueness follows directly from Theorem~\ref{thm:stability}.
			We shall prove the existence by a Picard iteration scheme and show convergence in $\mathcal{S}^2(-\infty,T;H)$.\\
			\noindent
			Set $X^{(0)}(t)=0$ for $t\in[0,T]$ and $X^{(0)}(t)=\gamma(t)$ for $t\leqslant 0$. For $n = 0,1,2,\ldots$, define the Picard sequence recursively by
			\begin{equation} \label{eq:Picard}
							X^{(n+1)}(t)=\gamma(0)+\int_0^t b(s,X^{(n)}_s)\,ds+\int_0^t \sigma(s,X^{(n)}_s)\,dW(s),
				\quad t\in[0,T],
			\end{equation}
			and $X^{(n+1)}(t)=\gamma(t)$ for $t\leqslant 0$. By Corollary \ref{cor:apriori}, we have $\mathbb{E}\big[\sup_{-\infty < t\leqslant T}\|X^{(n)}(t)\|_H^2\big]< +\infty$, and $X^{(n)}_t\in C_\lambda((-\infty,t];H)$ for each $n = 0,1,2,\ldots$, showing that $\big\{X^{(n)}(\cdot) \big\}_{n\geqslant 1}$ is a sequence in $\mathcal{S}_{\mathcal{F}}^2\left(-\infty,T ;H\right)$.
			Let $\Delta X^{(n)}(t):=X^{(n+1)}(t)-X^{(n)}(t)$, which satisfies the following equation
			\begin{align*}
				\Delta X^{(n)}(t)=&\int_{0}^t \Big[b\left(s, X^{(n)}_s\right) - b\left(s, X^{(n-1)}_s\right)\Big]d s\nonumber\\
				&+\int_{0}^t \Big[\sigma\left(s, X^{(n)}_s\right) - \sigma\left(s, X^{(n-1)}_s\right)\Big]d W(s).
			\end{align*}
			Applying It\^{o}'s formula (\ref{ito on scalar function}) to $\|\Delta X^{(n)}(t)\|_H^2$ on $t\in[0,T]$, we obtain
			\begin{align*}
				\|\Delta X^{(n)}(t)\|_H^2
				=&\;
				2\int_0^t
				\Big\langle \Delta X^{(n)}(s),
				b(s,X^{(n)}_s)-b(s,X^{(n-1)}_s)\Big\rangle_H\,ds \\
				&+2\int_0^t
				\Big\langle \Delta X^{(n)}(s),
				\big(\sigma(s,X^{(n)}_s)-\sigma(s,X^{(n-1)}_s)\big)\,dW(s)
				\Big\rangle_H \\
				&+\int_0^t
				\|\sigma(s,X^{(n)}_s)-\sigma(s,X^{(n-1)}_s)\|_{\mathcal{L}_2}^2\,ds .
			\end{align*}
			Taking the supremum over $t\in[0,T]$ and taking expectation yields
			\begin{align*}
				\mathbb{E}\Big[\sup_{0\leqslant t\leqslant T}\|\Delta X^{(n)}(t)\|_H^2\Big]
				\leqslant&
				2\mathbb{E}\bigg[\int_0^T
				\Big|\big\langle \Delta X^{(n)}(s),
				b(s,X^{(n)}_s)-b(s,X^{(n-1)}_s)\big\rangle_H\Big|ds\bigg] \\
				&+2\mathbb{E}\bigg[\sup_{0\leqslant t\leqslant T}
				\Big|\int_0^t
				\Big\langle \Delta X^{(n)}(s),
				(\sigma(s,X^{(n)}_s)-\sigma(s,X^{(n-1)}_s))dW(s)
				\Big\rangle_H\Big|\bigg] \\
				&+\mathbb{E}\bigg[\int_0^T
				\|\sigma(s,X^{(n)}_s)-\sigma(s,X^{(n-1)}_s)\|_{\mathcal{L}_2}^2ds\bigg] .
			\end{align*}
			For the drift coefficient, by Cauchy--Schwarz inequality and Young's inequality,
			\[
			2\Big|\big\langle \Delta X^{(n)}(s),
			b(s,X^{(n)}_s)-b(s,X^{(n-1)}_s)\big\rangle_H\Big|
			\leqslant \frac{1}{4T}\|\Delta X^{(n)}(s)\|_H^2
			+ 4T
			\|b(s,X^{(n)}_s)-b(s,X^{(n-1)}_s)\|_H^2 .
			\]
			For the stochastic integral term, by the Burkholder--Davis--Gundy inequality,
			\begin{align*}
				&\mathbb{E}\bigg[\sup_{0\leqslant t\leqslant T}
				\Big|\int_0^t
				\Big\langle \Delta X^{(n)}(s),
				(\sigma(s,X^{(n)}_s)-\sigma(s,X^{(n-1)}_s))dW(s)
				\Big\rangle_H\Big|\bigg] \\
				&\quad\leqslant
				C\mathbb{E}\Big(
				\int_0^T
				\|\Delta X^{(n)}(s)\|_H^2
				\|\sigma(s,X^{(n)}_s)-\sigma(s,X^{(n-1)}_s)\|_{\mathcal{L}_2}^2ds
				\Big)^{1/2} \\
				&\quad\leqslant
				\frac12\mathbb{E}\bigg[\sup_{0\leqslant s\leqslant T}\|\Delta X^{(n)}(s)\|_H^2\bigg]
				+ C\mathbb{E}\bigg[\int_0^T
				\|\sigma(s,X^{(n)}_s)-\sigma(s,X^{(n-1)}_s)\|_{\mathcal{L}_2}^2ds \bigg].
			\end{align*}
			Combining the above estimates, we obtain		
			\begin{align*}
			    \mathbb{E}&\Big[\sup_{0\leqslant t\leqslant T}\|\Delta X^{(n)}(t)\|_H^2\Big]\\
			&\leqslant C\,\mathbb{E}\bigg[\int_0^T
			\Big(\|b(s,X^{(n)}_s)-b(s,X^{(n-1)}_s)\|_H^2
			+\|\sigma(s,X^{(n)}_s)-\sigma(s,X^{(n-1)}_s)\|_{\mathcal{L}_2}^2\Big)ds \bigg].
			\end{align*}
			where $C>0$ is a constant independent of $n$. By Assumption \ref{ass:ISFDE-Lip} and by the fact that
			\[
			\|X^{(n)}_s-X^{(n-1)}_s\|_{\lambda,T}
			\leqslant \sup_{\theta\leqslant 0}\|X^{(n)}(s+\theta)-X^{(n-1)}(s+\theta)\|_H
			= \sup_{0\leqslant r\leqslant s}\|\Delta X^{(n-1)}(r)\|_H,
			\]
			we obtain
			\begin{equation}\label{eq:Picard-rec}
							\mathbb{E}\Big[\sup_{0\leqslant t\leqslant T}\|\Delta X^{(n)}(t)\|_H^2\Big]
				\leqslant C\int_0^T
				\mathbb{E}\Big[\sup_{0\leqslant r\leqslant s}\|\Delta X^{(n-1)}(r)\|_H^2\Big]\,ds,
			\end{equation}
			for some $C$ depending on $T$ and $L_1$.
			Define $k^n(t):=\mathbb{E}\big[\sup_{0\leqslant s\leqslant t}\|\Delta X^{(n)}(s)\|_H^2\big]$. Then \eqref{eq:Picard-rec} can be rewritten as
			\[
			k^n(T)\leqslant C\int_0^T k^{n-1}(s)ds.
			\]
			Iterating the above inequality yields that
			\[
			k^n(T)\leqslant \frac{(CT)^n}{n!}k^0(T),\quad n\geqslant 1,
			\]
			where we can easily show that $k^0(T)< +\infty$. Hence $\sum_{n\geqslant0}k^n(T)<\infty$ and $\big\{X^{(n)}(\cdot)\big\}_{n \geqslant 1}$ is a Cauchy sequence in $\mathcal{S}^2(-\infty,T;H)$.
			Therefore, there exists $X(\cdot)\in\mathcal{S}^2(-\infty,T;H)$ such that
			\[
			\mathbb{E}\Big[\sup_{0\leqslant t\leqslant T}\|X^{(n)}(t)-X(t)\|_H^2\Big]\rightarrow 0
			\quad\text{as }n\to\infty.
			\]
			Using Assumption \ref{ass:ISFDE-Lip} and the segment bound
			$\|X^{(n)}_t-X_t\|_{\lambda,T}\leqslant \sup_{0\leqslant s\leqslant t}\|X^{(n)}(s)-X(s)\|_H$ again,
			we have
			\[
			\mathbb{E}\Big[\int_0^T \|b(s,X^{(n)}_s)-b(s,X_s)\|_H^2\,ds\Big] \to 0,
			\qquad
			\mathbb{E}\Big[\int_0^T \|\sigma(s,X^{(n)}_s)-\sigma(s,X_s)\|_{\mathcal{L}_2}^2ds\Big] \to 0.
			\]
			Consequently,
			\[
			\int_0^t b(s,X^{(n)}_s)\,ds \to \int_0^t b(s,X_s)\,ds \quad \text{in }L^2(\Omega;H),
			\]
			and, by the It\^{o} isometry,
			\[
			\int_0^t \sigma(s,X^{(n)}_s)\,dW(s) \to \int_0^t \sigma(s,X_s)\,dW(s)\quad \text{in }L^2(\Omega;H),
			\]
			for all $t\in[0,T]$.
			Therefore, passing to the limit as $n\to \infty$ in \eqref{eq:Picard} yields that $X(\cdot)$ solves equation \eqref{ISDDE without control}.
		\end{proof}

\subsection{Infinitely anticipated backward stochastic evolution equations.}
		In this subsection, we study the well-posedness of the following Hilbert-valued infinitely anticipated backward stochastic evolution equation (IABSEE). It will serve as the adjoint equation to derive the maximum principle. 
		\begin{equation} \label{ABSDE without control}
			\left\{\begin{aligned}
				Y(t)=&\xi(T) + \int_{t}^{T}f\left(s, Y_{s+}, Z_{s+} \right) d t
				- \int_{t}^{T}Z(s) d W(s), \quad t \in[0, T]; \\
				Y(t) =&  \xi(t),\quad  Z(t) = \eta(t), \quad t \in[T, +\infty).
			\end{aligned}\right.
		\end{equation}
		where $\xi(\cdot) \in \mathcal{S}_{\mathcal{F}}^{2}(T,\infty;H)$ and $\eta(\cdot) \in \mathcal{M}_{\mathcal{F}}^{2,\beta}(T,\infty;\mathcal{L}_2(\mathbb{R}^m,H))$ are terminal conditions.
		For all $t\in [0,T]$, $f:\Omega \times \mathcal{M}_{\mathcal{F}}^{2}(t,\infty;H) \times \mathcal{M}_{\mathcal{F}}^{2,\beta}(t,\infty;\mathcal{L}_2(\mathbb{R}^m,H)) \rightarrow  L^2(\mathcal{F}_t,H)$ is the generator. 
        Moreover, we assume that for all
        $(y(\cdot),z(\cdot))\in \mathcal{M}_{\mathcal F}^{2}(t,\infty;H)
        \times \mathcal{M}_{\mathcal F}^{2,\beta}(t,\infty;\mathcal{L}_2(\mathbb{R}^m,H))$,
        the process
        $t\mapsto f(s,y_{t+},z_{t+})$
        is progressively measurable.
        Recall that $W$ is $m$-dimensional, then $Z(\cdot)$ takes values in $\mathcal{L}_2(\mathbb{R}^m,H)$. 

\begin{remark}
The generator in \eqref{ABSDE without control} depends on the whole future path segment and the pair $(Y,Z)$ is prescribed on $[T,\infty)$. In particular, the values of $(\xi,\eta)$ on $[T,\infty)$ make the anticipated terms $Y_{s+}$ and $Z_{s+}$ well defined for all $s\in[0,T]$.
\end{remark}
        
    \begin{definition}
A pair of processes $(Y(\cdot),Z(\cdot))$ is called a solution to IABSEE \eqref{ABSDE without control} if \eqref{ABSDE without control} holds $P$-a.s. and 
\[
(Y(\cdot),Z(\cdot))\in \mathcal{S}^2_{\mathcal F}(0,\infty;H)\times \mathcal{M}^{2,\beta}_{\mathcal F}(0,\infty;\mathcal{L}_2(\mathbb{R}^m,H)),
\]
\end{definition}
		
		In order to study the well-posedness of equation (\ref{ABSDE without control}), we impose the following assumptions:
		
		\begin{assumption}\label{ass:IABSDE-Lip}
			 (Uniform Lipschitz condition) There exists a constant $L_2>0$ such that for all $t \in [0, T]$, $\left(y(\cdot), z(\cdot)\right), \left(y^{\prime}(\cdot), z^{\prime}(\cdot)\right) \in \mathcal{M}_{\mathcal{F}}^{2}(t,\infty;H) \times \mathcal{M}_{\mathcal{F}}^{2,\beta}(t,\infty;\mathcal{L}_2(\mathbb{R}^m,H))$
		\begin{equation}\nonumber
			\begin{aligned}
			\mathbb{E}&\left[\int_{t}^{T}\left\|f(s,y_{s+},z_{s+})-f(s,y^{\prime}_{s+},z^{\prime}_{s+}) \right\|_{H}^2 e^{\beta s} ds\right]\\
			&\leqslant L_2\mathbb{E}\left[\int_{t}^{T}\sup_{s \leqslant r < +\infty}\left\|y(r) - y^{\prime}(r)\right\|_{H}^2 e^{\beta s} ds 
			+ \int_{t}^{+\infty}\left\|z(s) - z^{\prime}(s)\right\|_{\mathcal{L}_2}^2e^{\beta s} ds\right],
			\end{aligned}
		\end{equation}
		where $\beta \geqslant 0 $ is an arbitrary constant.
		\end{assumption}
		
		\begin{assumption}\label{ass:IABSDE-int}
		\begin{equation}\nonumber
		\mathbb{E} \left[\int_0^T\left\|f(s, 0,0)\right\|_{H}^2 d s\right] <+\infty.
		\end{equation}
		\end{assumption}

		We then have the following \textit{a priori} estimate of the solutions of IABSEE \eqref{ABSDE without control}.
\begin{theorem}\label{thm: IABSDE diff estimate}
	For $i=1,2$, let $f_i$ be the generators that satisfy Assumptions \ref{ass:IABSDE-Lip}--\ref{ass:IABSDE-int},  $\xi^{(i)}(\cdot) \in \mathcal{S}^2_{\mathcal{F}}\left(T, +\infty ; H\right)$, $\eta^{(i)}(\cdot) \in \mathcal{M}_{\mathcal{F}}^{2,\beta}\left(T, +\infty; \mathcal{L}_2(\mathbb{R}^m,H)\right)$ be the terminal conditions. Let $\big(Y^{(i)}(\cdot), Z^{(i)}(\cdot)\big)$ be the corresponding solutions to IABSEE \eqref{ABSDE without control}. Then there exists a constant $C>0$, depending only on $T$, $L_2$ and $\beta$, such that
		\begin{align}
		\mathbb{E}&\Big[\sup_{0\leqslant t< +\infty}\|Y^{(1)}(t)-Y^{(2)}(t)\|_H^2 + \int_{0}^{+\infty}e^{\beta t}\|Z^{(1)}(t) - Z^{(2)}(t)\|_{\mathcal{L}_2}^2 dt\Big] \nonumber \\
		&\leqslant C \bigg\{ \mathbb{E}\bigg[\sup_{T\leqslant t < +\infty}\|\xi^{(1)}(t)-\xi^{(2)}(t)\|_H^2 \bigg] 
		+ \mathbb{E}\bigg[\int_{T}^{+\infty}e^{\beta t}\|\eta^{(1)}(t) - \eta^{(2)}(t)\|_{\mathcal{L}_2}^2 dt \bigg] \nonumber \\
		&\qquad \quad+ \mathbb{E}\bigg[\int_0^T 
		\|f_1(t,Y^{(2)}_{t+}, Z^{(2)}_{t+})-f_2(t,Y^{(2)}_{t+}, Z^{(2)}_{t+})\|_H^2 dt \bigg] \bigg\}. \label{eq:IABSDE stability}
	\end{align}
\end{theorem}	
\begin{proof}
		For $t\in [0, +\infty)$, set
		\[
		\hat Y(t):=Y^{(1)}(t)-Y^{(2)}(t), \qquad
		\hat Z(t):=Z^{(1)}(t)-Z^{(2)}(t).
		\]
		Then $(\hat Y,\hat Z)$ satisfies
		\begin{equation} \nonumber
	\left\{\begin{aligned}
		\hat{Y}(t)=&\hat\xi(T) + \int_{t}^{T}\hat f\left(s \right) d t
		-\int_{t}^{T}\hat Z(s) d W(s), \quad t \in[0, T]; \\
		\hat Y(t) =& \hat \xi(t),\quad \hat Z(t) = \hat\eta(t), \quad t \in[T, +\infty).
	\end{aligned}\right.
\end{equation}
		where
		\[
		\hat f(s)
		:=f_1\bigl(s,Y^{(1)}_{s+},Z^{(1)}_{s+}\bigr)
		-f_2\bigl(s,Y^{(2)}_{s+},Z^{(2)}_{s+}\bigr),\quad \hat \xi(t): = \xi^{(1)}(t) - \xi^{(2)}(t), \quad \hat \eta(t): = \eta^{(1)}(t) - \eta^{(2)}(t).
		\]
		Applying It\^{o}'s formula (\ref{ito on scalar function}) to
		$e^{\beta t}\|\hat Y(t)\|_H^2$  on $[0,T]$ yields
		\begin{align}
			e&^{\beta t}\|\hat Y(t)\|_H^2 + \int_{t}^{T}e^{\beta s} \Big(\beta \|\hat Y(s)\|_H^2 + \|\hat Z(s)\|_{\mathcal{L}_2}^2\Big)ds \nonumber\\
			&= e^{\beta T}\|\hat Y(T)\|_H^2
			+2\int_{t}^{T}e^{\beta s}\langle \hat Y(s),\hat f(s)\rangle_H ds
			-2\int_{t}^{T}e^{\beta s}\langle \hat Y(s),\hat Z(s)dW(s)\rangle_H . \label{eq:IABSDE ito}
		\end{align}
		For the generator coefficient in (\ref{eq:IABDE diff esti generaotor}), by Assumption \ref{ass:IABSDE-Lip}, Cauchy--Schwarz inequality, Young's inequality, and the fact that $\sup_{s\leqslant r < +\infty}\|\hat Y(r)\|_{H}^2  \leqslant \sup_{s\leqslant r \leqslant T}\|\hat Y(r)\|_{H}^2 +\sup_{T\leqslant r < +\infty}\|\hat\xi(r)\|_H^2 $, we obtain 
		\begin{align}
			2&\mathbb{E}\Big[\int_{t}^{T}e^{\beta s}\langle \hat Y(s),\hat f(s)\rangle_H ds\Big] \nonumber \\
			&\leqslant \mathbb{E}\Big[\int_{t}^{T}e^{\beta s}\Big(2L_2\|\hat Y(s)\|_H^2 + \frac{1}{2L_2}\|\hat f(s)\|_H^2 \Big)ds\Big]\nonumber \\
			&\leqslant  \mathbb{E}\Big[\int_{t}^{T}e^{\beta s}\Big(2L_2\|\hat Y(s)\|_H^2 + \frac{1}{2}\sup_{s\leqslant r< +\infty} \|\hat Y(r)\|_H^2\Big)ds\Big] 
			\nonumber \\
			&\quad+ \frac{1}{2}\mathbb{E}\Big[\int_{t}^{+\infty}e^{\beta s} \|\hat Z(s)\|_{\mathcal{L}_2}^2 ds\Big]
			+\frac{1}{2} \mathbb{E}\Big[\int_{t}^{T}e^{\beta s}\|f_1(s,Y^{(2)}_{s+}, Z^{(2)}_{s+})-f_2(s,Y^{(2)}_{s+}, Z^{(2)}_{s+})\|_H^2 ds\Big]\nonumber \\
			&\leqslant 2L_2 \mathbb{E}\Big[\int_{t}^{T}e^{\beta s}\|\hat Y(s)\|_H^2ds\Big] 
			+ \frac{1}{2}\mathbb{E}\Big[\int_{t}^{T}\sup_{s\leqslant r\leqslant T}\big(e^{\beta r} \|\hat Y(r)\|_H^2\big) ds\Big]
			+ C\mathbb{E}\Big[\sup_{T\leqslant r < +\infty}\|\hat\xi(r)\|_H^2 \Big] \nonumber \\
			&\quad+\frac{1}{2}\mathbb{E}\Big[\int_{t}^{T}e^{\beta s} \|\hat Z(s)\|_{\mathcal{L}_2}^2 ds\Big] 
			+C\mathbb{E}\Big[\int_{T}^{+\infty}e^{\beta s} \|\hat \eta(s)\|_{\mathcal{L}_2}^2 ds\Big]\nonumber \\
			&\quad+ C\mathbb{E}\Big[\int_{t}^{T}\|f_1(s,Y^{(2)}_{s+}, Z^{(2)}_{s+})-f_2(s,Y^{(2)}_{s+}, Z^{(2)}_{s+})\|_H^2 ds\Big]. \label{eq:IABDE diff esti generaotor}
		\end{align}
		Taking expectation on both sides of (\ref{eq:IABSDE ito}) and combining (\ref{eq:IABDE diff esti generaotor}), we have the following estimate over the $Z$ component by choosing $\beta > 2L_2$,
		\begin{align}
			\mathbb{E}\bigg[\int_{t}^{T}e^{\beta s} \|\hat Z(s)\|_{\mathcal{L}_2}^2ds\bigg]&\leqslant  C\bigg\{\mathbb{E}\Big[\int_{t}^{T}\sup_{s\leqslant r\leqslant T}\big(e^{\beta r} \|\hat Y(r)\|_H^2\big) ds\Big] \nonumber \\
			&\quad\quad+ 
			\mathbb{E}\Big[\sup_{T\leqslant t < +\infty}\|\hat\xi(t)\|_H^2\Big] 
			+\mathbb{E}\Big[ \int_{T}^{+\infty}e^{\beta t}\left\|\hat \eta(t)\right\|_{\mathcal{L}_2}^2 ds\Big] 
			 \nonumber \\
			&\quad\quad+ \mathbb{E}\Big[\int_{t}^{T}\|f_1(s,Y^{(2)}_{s+}, Z^{(2)}_{s+})-f_2(s,Y^{(2)}_{s+}, Z^{(2)}_{s+})\|_H^2 ds\Big]\bigg\}. \label{IABSDE Z est}
		\end{align}
		Moreover, for the martingale term in (\ref{eq:IABSDE ito}), notice that for all $s\in [t,T]$,
		\begin{align*}
			\left|\int_s^T e^{\beta r}\langle \hat Y(r), \hat Z(r) d W(r)\rangle_H\right| & =\left|\int_t^T e^{\beta r}\langle\hat Y(r), \hat Z(r) d W(r)\rangle_H-\int_t^s e^{\beta r}\langle\hat Y(r), \hat Z(r) d W(r)\rangle_H\right| \\
			& \leqslant\left|\int_t^T e^{\beta r}\langle\hat Y(r), \hat Z(r) d W(r)\rangle_H\right|+\left|\int_t^s e^{\beta r}\langle\hat Y(r), \hat Z(r) d W(r)\rangle_H\right|,
		\end{align*}
		by the Burkholder--Davis--Gundy inequality, we have for some constant $C>0$,
		\begin{align}
			\mathbb{E}&\left[\sup_{t\leqslant s\leqslant T}\left|\int_{s}^{T}e^{\beta r}\langle \hat Y(r),\hat Z(r)dW(r)\rangle_H \right|\right]\nonumber \\
			&\leqslant 2\mathbb E\left[\sup _{t \leqslant s \leqslant T}\left|\int_t^s e^{\beta r}\langle \hat Y(r),\hat Z(r)dW(r)\rangle_H\right|\right] \nonumber \\
			& \leqslant C\mathbb E\left[\left(\int_t^T e^{2 \beta r}\|\hat Y(r)\|_H^2 \|\hat Z(r)\|_{\mathcal{L}_2}^2 d r\right)^{1 / 2}\right] \nonumber \\
			& \leqslant C\mathbb E\left[\left(\sup _{t \leqslant r \leqslant T} e^{\beta r / 2}\|\hat Y(r)\|_H^2\right)\left(\int_t^T e^{\beta r}\|\hat Z(r)\|_{\mathcal{L}_2}^2 d r\right)^{1 / 2}\right] \nonumber\\
			& \leqslant \frac{1}{4}\mathbb E\left[\sup _{t \leqslant r \leqslant T} e^{\beta r}\|\hat Y(r)\|_H^2\right]+ C E\left[\int_t^T e^{\beta r}\|\hat Z(r)\|_{\mathcal{L}_2}^2 d r\right]. \label{IABSDE BDG}
		\end{align}
		Taking supremum for  $t\in[\tau,T]$   for $\tau \in [0,T)$, then taking expectation on both sides of (\ref{eq:IABSDE ito}), and combining (\ref{eq:IABDE diff esti generaotor}) (\ref{IABSDE Z est}) (\ref{IABSDE BDG}), we obtain by choosing $\beta > \frac{1}{2}$,
		\begin{align*}
			\mathbb{E}\left[\sup_{\tau\leqslant s\leqslant T}e^{\beta t}\| \hat Y(t)\|^2_H\right] &\leqslant  C\bigg\{\mathbb{E}\Big[\int_{\tau}^{T}\sup_{s\leqslant r\leqslant T}\big(e^{\beta r} \|\hat Y(r)\|_H^2\big) ds\Big] \nonumber \\
				&\quad\quad+ 
				\mathbb{E}\Big[\sup_{T\leqslant t < +\infty}\|\hat\xi(t)\|_H^2\Big] 
				+\mathbb{E}\Big[ \int_{T}^{+\infty}e^{\beta t}\left\|\hat \eta(t)\right\|_{\mathcal{L}_2}^2 ds\Big] 
				\nonumber \\
				&\quad\quad+ \mathbb{E}\Big[\int_{\tau}^{T}\|f_1(s,Y^{(2)}_{s+}, Z^{(2)}_{s+})-f_2(s,Y^{(2)}_{s+}, Z^{(2)}_{s+})\|_H^2 ds\Big]\bigg\}.
			\end{align*}
		Applying Gr\"onwall's inequality,  we have
		\begin{align*}
			\mathbb{E}\Big[
			\sup_{0\leqslant t\leqslant T} e^{\beta t}\|\hat Y(t)\|_H^2
			+\int_0^T e^{\beta t}\|\hat Z(t)\|_{\mathcal{L}_2}^2dt
			\Big]
			\leqslant&
			C\mathbb{E}\Big[
			\sup_{T\leqslant t < +\infty}\|\hat\xi(t)\|_H^2
			+\int_T^{+\infty} e^{\beta t}\|\hat \eta(t)\|_{\mathcal{L}_2}^2
			\bigr)dt
			\Big]\\
			&+ \mathbb{E}\Big[\int_{0}^{T}\|f_1(t,Y^{(2)}_{t+}, Z^{(2)}_{t+})-f_2(t,Y^{(2)}_{t+}, Z^{(2)}_{t+})\|_H^2 dt\Big].
		\end{align*}
	Then the desired stability estimate (\ref{eq:IABSDE stability}) follows.
	
\end{proof}

\begin{corollary}\label{^{(n)}}
	Suppose Assumptions \ref{ass:IABSDE-Lip}--\ref{ass:IABSDE-int} hold, let $\big(Y(\cdot),Z(\cdot)\big)$ be the solution to IABSEE \eqref{ABSDE without control}. Then for any given terminal conditions $\xi(\cdot) \in \mathcal{S}^2_{\mathcal{F}}\left(T, +\infty ; H\right)$, $\eta(\cdot) \in \mathcal{M}_{\mathcal{F}}^{2,\beta}\left(T, +\infty; \mathcal{L}_2(\mathbb{R}^m,H)\right)$, there exists a constant $C>0$, depending only on $T$, $L_2$ and $\beta$, such that
	\begin{align*}
		\mathbb{E}&\Big[\sup_{0\leqslant t\leqslant T}\|Y(t)\|_H^2 + \int_{0}^{+\infty}e^{\beta t}\left\|Z(t)\right\|_{\mathcal{L}_2}^2dt\Big] \\
		&\leqslant C \bigg\{ \mathbb{E}\Big[\sup_{T\leqslant t < 	+\infty}\|\xi(t)\|_H^2 \Big] 
		+ \mathbb{E}\Big[\int_{T}^{+\infty}e^{\beta t}\| \eta(t)\|_{\mathcal{L}_2}^2 dt \Big] 
		+ \mathbb{E}\Big[\int_0^T 
		\|f(t,0,0)\|_H^2 dt \Big] \bigg\}. 
	\end{align*}
	Moreover, we have $(Y(\cdot), Z(\cdot)) \in \mathcal{S}_{\mathcal{F}}^2\left(0, +\infty; H\right) \times \mathcal{M}_{\mathcal{F}}^{2,\beta}\left(0, +\infty; \mathcal{L}_2(\mathbb{R}^m,H)\right)$ and $\left\{f\left(t, Y_{t+}, Z_{t+}\right)\right\}_{t \in[0, T]} \in \mathcal{M}_{\mathcal{F}}^2(0, T ; H)$.
\end{corollary}

\medskip

		Now we study the well-posedness of the IABSEE (\ref{ABSDE without control}). 
		\begin{theorem}\label{thm:IABSDE-wellposed}
			Suppose Assumptions \ref{ass:IABSDE-Lip}-\ref{ass:IABSDE-int} hold.
			Then for any given terminal conditions $\xi(\cdot) \in \mathcal{S}^2_{\mathcal{F}}\left(T, +\infty ; H\right)$, $\eta(\cdot) \in \mathcal{M}_{\mathcal{F}}^{2,\beta}\left(T, +\infty; \mathcal{L}_2(\mathbb{R}^m,H)\right)$, IABSEE (\ref{ABSDE without control}) admits a unique solution, that is, there exists a unique pair of processes $(Y(\cdot), Z(\cdot))\in \mathcal{S}_\mathcal{F}^2\left(0, +\infty; H\right) \times \mathcal{M}_\mathcal{F}^{2,\beta}\left(0, +\infty; \mathcal{L}_2(\mathbb{R}^m,H)\right)$ satisfying (\ref{ABSDE without control}).
		\end{theorem}
\begin{proof}
Uniqueness follows directly from the Stability Theorem~\ref{thm: IABSDE diff estimate}.
We shall prove the existence through the Picard iteration.\\
\noindent
Define the initial pair $\big(Y^{(0)}(t),Z^{(0)}(t)\big)=\big(0,0\big)$ for $t\in[0,T]$ and $\big(Y^{(0)}(t),Z^{(0)}(t)\big)=\big(\xi(t),\eta(t)\big)$ for $t\geqslant T$. Obviously, $(Y^{(0)}(\cdot), Z^{(0)}(\cdot))\in \mathcal{S}_\mathcal{F}^2\left(0, +\infty; H\right) \times \mathcal{M}_\mathcal{F}^{2,\beta}\left(0, +\infty; \mathcal{L}_2(\mathbb{R}^m,H)\right)$. For $n=1,2, \ldots$, define the Picard sequence through
\begin{equation}\label{picard}
	\left\{ \begin{aligned}
		Y^{(n)}(t)=&\xi(T)+\int_{t}^T f\big(s, Y_{s+}^{(n-1)}, Z_{s+}^{(n-1)}\big) d s-\int_{t}^T Z^{(n)}(s) d W(s), \quad &&t\in [0,T];\\
Y^{(n)}(t) =&  \xi(t),\quad  Z^{(n)}(t) = \eta(t), \quad t \in[T, +\infty).
	\end{aligned}\right.
\end{equation}
By Corollary (\ref{^{(n)}}), we have a uniform \textit{a priori} bound for the Picard sequence 
\begin{align*}
	\mathbb{E}&\Big[\sup_{0\leqslant t\leqslant T}\|Y^{(n)}(t)\|_H^2 + \int_{0}^{+\infty}e^{\beta t}\left\|Z^{(n)}(t)\right\|_{\mathcal{L}_2}^2dt\Big] \\
	&\leqslant C \bigg\{ \mathbb{E}\Big[\sup_{T\leqslant t < 	+\infty}\|\xi(t)\|_H^2 \Big] 
	+ \mathbb{E}\Big[\int_{T}^{+\infty}e^{\beta t}\| \eta(t)\|_{\mathcal{L}_2}^2 dt \Big] 
	+ \mathbb{E}\bigg[\int_0^T 
	\|f(t,0,0)\|_H^2 dt \bigg] \bigg\}. 
\end{align*}
where the constant $C>0$ is independent of $n$. Thus, $\big\{\big(Y^{(n)}(\cdot), Z^{(n)}(\cdot)\big) \big\}_{n \geqslant 1}$ is a sequence in $\mathcal{S}_{\mathcal{F}}^2\left(0,+\infty ; H\right) \times\mathcal{M}_{\mathcal{F}}^{2,\beta}\left(0,+\infty ;\mathcal{L}_2(\mathbb{R}^m,H)\right)$.\\
\noindent
Consider the equation for the difference $\Delta Y^{(n)}(t):=Y^{(n)}(t)-Y^{(n-1)}(t)$, $\Delta Z^{(n)}(t):=Z^{(n)}(t)-Z^{(n-1)}(t)$,
\begin{align*}
	\Delta Y^{(n)}(t) = \int_{t}^{T}\big[f(s, Y^{(n-1)}_{s+},  Z^{(n-1)}_{s+}) - f(s, Y^{(n-2)}_{s+},  Z^{(n-2)}_{s+})\big]ds - \int_{t}^{T} \Delta Z^{(n)}(s) dW(s), \quad t\in[0,T].
\end{align*}
with the terminal conditions $Y^{(n)}(t) = 0, Z^{(n)}(t) = 0$ for $t\geqslant T$.
Applying It\^{o}'s formula (\ref{ito on scalar function}) to
$e^{\beta t}\|\Delta Y^{(n)}(t)\|_H^2$  on $[0,T]$ yields
\begin{align}
	e&^{\beta t}\|\Delta Y^{(n)}(t)\|_H^2 + \int_{t}^{T}e^{\beta s} \Big(\beta \|\Delta Y^{(n)}(s)\|_H^2 + \|\Delta Z^{(n)}(s)\|_{\mathcal{L}_2}^2\Big)ds \nonumber\\
	&= 2\int_{t}^{T}e^{\beta s}\langle \Delta Y^{(n)}(s),f(s, Y^{(n-1)}_{s+},  Z^{(n-1)}_{s+}) - f(s, Y^{(n-2)}_{s+},  Z^{(n-2)}_{s+})\rangle_H ds \nonumber \\
	&\quad +2\int_{t}^{T}e^{\beta s}\langle \Delta Y^{(n)}(s),\Delta Z^{(n)}(s)dW(s)\rangle_H . \label{ito}
\end{align}
Using the same method we deal with (\ref{eq:IABSDE ito}) to analyze the generator and the martingale part in (\ref{ito}), we obtain that there exists a positive constant $C$ depending only on $T$, $L_2$ and $\beta$ such that 
\begin{align*}
	\mathbb{E} & \Big[\sup_{t\leqslant s\leqslant T}e^{\beta s}\|\Delta Y^{(n)}(s)\|_H^2\Big] +
	\mathbb{E} \Big[ \int_t^T e^{\beta s}\|\Delta Z^{(n)}(s)\|^2_{\mathcal{L}_2} ds\Big]\nonumber\\
	& \leqslant C \int_t^T \mathbb{E} \Big[\sup_{s\leqslant r\leqslant T}e^{\beta r}\|\Delta Y^{(n-1)}(r)\|_H^2\Big] d s.
\end{align*}
Let $u^n(t):=\mathbb{E} \Big[\sup_{t\leqslant s\leqslant T}e^{\beta s}\|\Delta Y^{(n)}(s)\|_H^2\Big]$ and
$v^n(t):=\mathbb{E} \Big[ \int_t^T e^{\beta s}\|\Delta Z^{(n)}(s)\|^2_{\mathcal{L}_2} ds\Big]$, the above formula can be rewritten as
\begin{equation}\nonumber
	u^n(0)+v^n(0)\leqslant C\int_{0}^{T}u^{n-1}(s) d s, \quad u^{n}(T)=0,\quad n\geqslant 1.
\end{equation}
Iterating the above inequality yields that
\[
u^n(0)\leqslant \frac{(TC)^{n-1}}{(n-1)!}u^1(0), \quad
v^n(0)\leqslant \frac{T^{n-1}C^{n-2}}{(n-2)!}u^1(0),
\]
where we can easily show that  $u^1(0)< +\infty$. Thus $\left\{\left(Y^{(n)}(\cdot), Z^{(n)}(\cdot)\right)\right\}_{n \geqslant 1}$ is a Cauchy sequence in Banach space $\mathcal{S}_{\mathcal{F}}^2\left(0,+\infty ; H\right) \times$ $\mathcal{M}_{\mathcal{F}}^{2,\beta}\left(0,+\infty ; \mathcal{L}_2(\mathbb{R}^m,H)\right)$.  \\
\noindent
Denote by $\big(Y(\cdot), Z(\cdot)\big)$ the limit of Picard sequence $\left\{\left(Y^{(n)}(\cdot), Z^{(n)}(\cdot)\right)\right\}_{n \geqslant 1}$ as  $n \rightarrow \infty$. By Assumption \ref{ass:IABSDE-Lip}, we have 
\[
\mathbb{E}\Big[\int_0^T \|f(s,Y^{(n)}_{s+}, Z^{(n)}_{s+})-f(s,Y_{s+}, Z_{s+})\|_H^2\,ds\Big] \to 0.
\]
Then for all $t\in[0,T]$, the following convergence holds in $L^2(\Omega;H)$, 
\[
\int_t^T f(s,Y^{(n)}_{s+}, Z^{(n)}_{s+})ds \to \int_t^T f(s,Y_{s+}, Z_{s+})ds, \quad 
\int_{t}^T Z^{(n)}(s) d W(s) \to \int_{t}^T Z(s) d W(s).
\]
Passing the limit in (\ref{picard}), as $n \rightarrow \infty$, we obtain that $\big(Y(\cdot), Z(\cdot)\big)$ 
solves equation (\ref{ABSDE without control}).

\end{proof}

       \section{Non-anticipative path derivatives and infinite-window dual operators in Hilbert spaces.}\label{chapter Non-anticipative path derivatives and infinite-window dual operators in Hilbert spaces}

This section gives an infinite-window Hilbert-space version of the delay/anticipation duality, which will play a key role in deriving the adjoint equation and the SMP in the following chapter.
Throughout this section, $E$ and $F$ are generic separable Hilbert spaces with inner products $\langle\cdot,\cdot\rangle_E$ and $\langle\cdot,\cdot\rangle_F$ respectively, and $\alpha$ is the finite signed Borel measure on $[0,+\infty)$. We denote by $|\alpha|$ the total variation of $\alpha$, and we fix a Borel function
\[
\jmath:[0,+\infty)\to\{-1,1\}
\quad\text{such that}\quad
\alpha(d\theta)=\jmath(\theta)\,|\alpha|(d\theta).
\]

No absolute continuity with respect to the Lebesgue measure is imposed on $\alpha$; in particular, Dirac cases and mixed discrete-continuous measures are fully allowed.

\subsection{Fading-memory paths and lag-kernel derivatives.}\label{section 5.1}

Fix $\lambda>0$ and recall the definition of the space
\[
C_\lambda(( -\infty,T];E)
:=
\left\{x\in C(( -\infty,T];E):
\|x\|_{\lambda,T}:=\sup_{s\leqslant T}e^{\lambda s}\|x(s)\|_E<\infty,
\ \lim_{s\to-\infty}e^{\lambda s}x(s)=0
\right\}.
\]

For each $t\in[0,T]$, define the non-anticipative extension
\[
x_t(s):=
\begin{cases}
x(s),& s\leqslant t,\\
x(t),& t<s\leqslant T,
\end{cases}
\qquad s\in(-\infty,T],
\]
and the non-anticipative subspace
\[
C_{\lambda,t}(( -\infty,T];E):=\{x_t:x\in C_\lambda(( -\infty,T];E)\}.
\]

\begin{lemma}\label{lem:fading-memory-banach}
The space $C_\lambda(( -\infty,T];E)$ is a Banach space. Moreover, for every $r\leqslant T$, the evaluation map
\[
\delta_r:C_\lambda(( -\infty,T];E)\to E,
\qquad \delta_r(x):=x(r),
\]
is continuous and satisfies
\[
\|x(r)\|_E\leqslant e^{-\lambda r}\|x\|_{\lambda,T},
\qquad x\in C_\lambda(( -\infty,T];E).
\]
\end{lemma}

\begin{proof}
Let $(x_n)_{n\geqslant1}$ be a Cauchy sequence in $C_\lambda(( -\infty,T];E)$, and define
\[
y_n(s):=e^{\lambda s}x_n(s),\qquad s\leqslant T.
\]
Then $(y_n)_{n\geqslant1}$ is Cauchy under the sup norm on $C(( -\infty,T];E)$, hence there exists a bounded continuous function $y:( -\infty,T]\to E$ such that
\[
\sup_{s\leqslant T}\|y_n(s)-y(s)\|_E\longrightarrow0.
\]
We claim that $y(s)\to0$ as $s\to-\infty$. Fix $\varepsilon>0$. Choose $N$ such that
\[
\sup_{s\leqslant T}\|y_N(s)-y(s)\|_E<\frac{\varepsilon}{2}.
\]
Since $x_N\in C_\lambda(( -\infty,T];E)$, there exists $R<0$ such that
\[
\|y_N(s)\|_E=\|e^{\lambda s}x_N(s)\|_E<\frac{\varepsilon}{2},
\qquad s\leqslant R.
\]
Hence, for all $s\leqslant R$,
\[
\|y(s)\|_E\leqslant \|y(s)-y_N(s)\|_E+\|y_N(s)\|_E<\varepsilon.
\]
Thus $y(s)\to0$ as $s\to-\infty$.
Define
\[
x(s):=e^{-\lambda s}y(s),\qquad s\leqslant T.
\]
Then $x\in C_\lambda(( -\infty,T];E)$ and
\[
\|x_n-x\|_{\lambda,T}
=
\sup_{s\leqslant T}e^{\lambda s}\|x_n(s)-x(s)\|_E
=
\sup_{s\leqslant T}\|y_n(s)-y(s)\|_E
\longrightarrow0.
\]
Therefore $C_\lambda(( -\infty,T];E)$ is complete.
\\
\\
\noindent
For the evaluation map $\delta_r$, the definition of $\|\cdot\|_{\lambda,T}$ gives, for every $r\leqslant T$,
\[
\|x(r)\|_E=e^{-\lambda r}e^{\lambda r}\|x(r)\|_E
\leqslant e^{-\lambda r}\sup_{s\leqslant T}e^{\lambda s}\|x(s)\|_E
=e^{-\lambda r}\|x\|_{\lambda,T}.
\]
Hence $\delta_r$ is continuous.
\end{proof}

\begin{definition} \label{def: non-anticipatively differentiable}
Let
\[
a:[0,T]\times C_\lambda(( -\infty,T];E)\to F
\]
be measurable. We say that $a$ is \emph{non-anticipatively differentiable with respect to the path variable} at $(t,x)$ if there exists a bounded linear operator
\[
\partial_x a(t,x)\in \mathcal L\bigl(C_{\lambda,t}(( -\infty,T];E),F\bigr)
\]
such that
\[
a(t,x+h)-a(t,x)-\partial_x a(t,x)h=o(\|h\|_{\lambda,T})
\qquad\text{as }\|h\|_{\lambda,T}\to0,
\]
for all $h\in C_{\lambda,t}(( -\infty,T];E)$.
We call the operator $\partial_x a(t,x)$ the (infinite-window) non-anticipative path derivative of $a(t,\cdot)$ at the path $x$. 
\end{definition}

If such a derivative $\partial_x a(t,x)$ exists, it is unique. Indeed, the non-anticipative path derivative is precisely the Fr\'echet derivative restricted to the past information over $(-\infty,t]$. The uniqueness of $\partial_x a(t,x)$ therefore follows directly from the uniqueness of the Fr\'echet derivative; see Remark~\ref{connection of 2 derivatives} for further details.

\begin{assumption}\label{ass:lag-kernel}
Fix a reference path $x\in C_\lambda(( -\infty,T];E)$, and write
\[
\rho_t:=\partial_x a(t,x)
\in \mathcal L\bigl(C_{\lambda,t}(( -\infty,T];E),F\bigr),
\qquad t\in[0,T].
\]
We assume that there exists a kernel
\[
K:[0,T]\times[0,+\infty)\to\mathcal L(E,F)
\]
such that:
\begin{enumerate}
\item[(i)] for every $e\in E$, the map $(t,\theta)\mapsto K(t,\theta)e$ is jointly measurable;
\item[(ii)] for every $t\in[0,T]$ and every $h\in C_{\lambda,t}(( -\infty,T];E)$,
\[
\rho_t(h)=\int_0^{+\infty}K(t,\theta)h(t-\theta)\,\alpha(d\theta);
\]
\item[(iii)] the constants
\[
M_0:=\sup_{t\in[0,T]}\int_0^{+\infty}\|K(t,\theta)\|_{\mathcal L(E,F)}\,|\alpha|(d\theta)
\]
and
\[
M:=\sup_{u\in(-\infty,T]}\int_0^{+\infty}\|K(u+\theta,\theta)\|_{\mathcal L(E,F)}\mathbf 1_{[0,T]}(u+\theta)\,|\alpha|(d\theta)
\]
are finite.
\end{enumerate}
\end{assumption}

\begin{remark}
The non-anticipative operator in the sense of Definition~\ref{def: non-anticipatively differentiable} is the infinite-window extension of the non-anticipative derivative introduced in Liu, Song and Wang \cite{songjian2025}. We note that Definition~\ref{def: non-anticipatively differentiable} allows arbitrary finite signed measures $\alpha$, in particular:
\begin{enumerate}
\item[(a)] if $\alpha=\delta_\delta$, then
\[
\rho_t(h)=K(t,\delta)h(t-\delta);
\]
The Dirac case is well defined. Indeed, by Lemma~\ref{lem:fading-memory-banach}, the evaluation map $h\mapsto h(t-\delta)$ is continuous on $C_\lambda(( -\infty,T];E)$.
\item[(b)] if $\alpha$ is absolutely continuous with density $w(\theta)$, then
\[
\rho_t(h)=\int_0^{+\infty}K(t,\theta)h(t-\theta)w(\theta)\,d\theta.
\]
\end{enumerate}
\end{remark}

\begin{remark}
The finite-dimensional delay operator is covered by taking $E=F=\mathbb R^k$ and
\[
K(t,\theta)=\phi(t-\theta,t)I_k.
\]
Then
\[
\rho_t(h)=\int_0^\infty \phi(t-\theta,t)\,h(t-\theta)\,\alpha(d\theta),
\]
which coincides with the delay control operator studied in Cheng \cite{ziji2}. The Hilbert-valued delay control operator \(v_d\) in \eqref{ISDDE with control} is treated in Corollary~\ref{cor:section3-specialization} by taking \(E=F=U\).
\end{remark}

Now we present a  change-of-variables identity, which will be used repeatedly to prove the duality identity in the following chapters. 
\begin{lemma}\label{lem:cv-alpha}
Let $g:[0,T]\times[0,+\infty)\to[0,+\infty]$ be measurable. Then
\begin{equation}\label{eq:cv-alpha}
\int_0^T\int_0^{+\infty}g(t,\theta)\,|\alpha|(d\theta)dt
=
\int_{-\infty}^T\int_0^{+\infty}g(u+\theta,\theta)\mathbf 1_{[0,T]}(u+\theta)\,|\alpha|(d\theta)du.
\end{equation}
\end{lemma}

\begin{proof}
By Tonelli's theorem,
\[
\int_0^T\int_0^{+\infty}g(t,\theta)\,|\alpha|(d\theta)dt
=
\int_0^{+\infty}\int_0^T g(t,\theta)\,dt\,|\alpha|(d\theta).
\]
Fix $\theta\geqslant0$. With the substitution $u=t-\theta$, so that $t=u+\theta$ and $dt=du$, we obtain
\[
\int_0^T g(t,\theta)\,dt
=
\int_{-\theta}^{T-\theta}g(u+\theta,\theta)\,du
=
\int_{-\infty}^T g(u+\theta,\theta)\mathbf 1_{[-\theta,T-\theta]}(u)\,du.
\]
Since
\[
\mathbf 1_{[-\theta,T-\theta]}(u)=\mathbf 1_{[0,T]}(u+\theta),
\]
we get
\[
\int_0^T g(t,\theta)\,dt
=
\int_{-\infty}^T g(u+\theta,\theta)\mathbf 1_{[0,T]}(u+\theta)\,du.
\]
Integrating this identity with respect to $|\alpha|(d\theta)$ and applying Tonelli's theorem once more yields \eqref{eq:cv-alpha}.
\end{proof}

\subsection{The induced \texorpdfstring{$L^2$}{L2}-operator and its Hilbert adjoint.}\label{The induced L2-operator and its Hilbert adjoint}

We now pass from the pathwise derivative to the associated infinite-window operator on Hilbert spaces. For $Z\in L^2( -\infty,T;E)$, define
\begin{equation}\label{eq:R-def}
(\mathcal RZ)(t)
:=
\int_0^{+\infty}K(t,\theta)Z(t-\theta)\,\alpha(d\theta),
\qquad t\in[0,T],
\end{equation}
whenever the integral is well defined.

\begin{theorem}\label{thm:R-bounded}
Under Assumption~\ref{ass:lag-kernel}, the following statements hold.
\begin{enumerate}
\item[(i)] For every $Z\in L^2( -\infty,T;E)$, the right-hand side of \eqref{eq:R-def} is well defined for almost every $t\in[0,T]$ as a Bochner integral in $F$.
\item[(ii)] The map $Z\mapsto \mathcal RZ$ is well defined on $L^2( -\infty,T;E)$; namely, if $Z_1=Z_2$ almost everywhere on $( -\infty,T]$, then $\mathcal RZ_1=\mathcal RZ_2$ almost everywhere on $[0,T]$.
\item[(iii)] $\mathcal R$ is a bounded linear operator from $L^2(-\infty,T;E)$ to $L^2(0,T;F)$ and satisfies
\begin{equation}\label{eq:R-bound}
\|\mathcal RZ\|_{L^2(0,T;F)}^2
\leqslant M_0M\,\|Z\|_{L^2(-\infty,T;E)}^2,
\qquad Z\in L^2(-\infty,T;E).
\end{equation}
\end{enumerate}
\end{theorem}

\begin{proof}
Let
\[
\widetilde K(t,\theta):=\jmath(\theta)K(t,\theta),
\qquad (t,\theta)\in[0,T]\times[0,+\infty).
\]
Then
\[
\|\widetilde K(t,\theta)\|_{\mathcal L(E,F)}=\|K(t,\theta)\|_{\mathcal L(E,F)}
\]
and
\[
\int_0^{+\infty}K(t,\theta)Z(t-\theta)\,\alpha(d\theta)
=
\int_0^{+\infty}\widetilde K(t,\theta)Z(t-\theta)\,|\alpha|(d\theta).
\]
Thus, it is enough to work with the positive measure $|\alpha|$.
\\
\\
\indent
We start by proving (i). Fix $Z\in L^2(-\infty,T;E)$. Choose a strongly measurable representative of $Z$, still denoted by $Z$. Since $E$ is separable, there exists a sequence of simple measurable functions $(Z_n)_{n\geqslant1}$ such that $Z_n(s)\to Z(s)$ for almost every $s\leqslant T$. Modifying $Z$ on a Lebesgue-null set if necessary, we may and do assume that $Z_n(s)\to Z(s)$ for every $s\leqslant T$.
\\
\noindent
For each $n$, write
\[
Z_n(s)=\sum_{m=1}^{N_n}e_{n,m}\mathbf 1_{A_{n,m}}(s),
\]
where $e_{n,m}\in E$ and $A_{n,m}\subset(-\infty,T]$ are measurable. Then
\[
(t,\theta)\longmapsto \widetilde K(t,\theta)Z_n(t-\theta)
=
\sum_{m=1}^{N_n}\widetilde K(t,\theta)e_{n,m}\,\mathbf 1_{A_{n,m}}(t-\theta)
\]
is jointly measurable because $(t,\theta)\mapsto \widetilde K(t,\theta)e_{n,m}$ is measurable by Assumption~\ref{ass:lag-kernel} (i), and $(t,\theta)\mapsto t-\theta$ is measurable. Passing to the pointwise limit in $n$, we conclude that
\[
(t,\theta)\longmapsto \widetilde K(t,\theta)Z(t-\theta)
\]
is jointly measurable.
Define the nonnegative measurable function
\[
g(t,\theta):=\|Z(t-\theta)\|_E^2\,\|K(t,\theta)\|_{\mathcal L(E,F)},
\qquad (t,\theta)\in[0,T]\times[0,+\infty).
\]
Applying Lemma~\ref{lem:cv-alpha} to $g$, we obtain
\begin{align*}
\int_0^T\int_0^{+\infty}g(t,\theta)\,|\alpha|(d\theta)dt
&=
\int_{-\infty}^T\int_0^{+\infty}\|Z(u)\|_E^2\,\|K(u+\theta,\theta)\|_{\mathcal L(E,F)}\mathbf 1_{[0,T]}(u+\theta)\,|\alpha|(d\theta)du\\
&\leqslant
M\int_{-\infty}^T\|Z(u)\|_E^2\,du
<\infty.
\end{align*}
Therefore, by Tonelli's theorem,
\begin{equation}\label{eq:g-finite}
\int_0^{+\infty}\|Z(t-\theta)\|_E^2\,\|K(t,\theta)\|_{\mathcal L(E,F)}\,|\alpha|(d\theta)<\infty
\end{equation}
for almost every $t\in[0,T]$.
Fix such a $t$. Since
\[
\|\widetilde K(t,\theta)Z(t-\theta)\|_F
\leqslant
\|K(t,\theta)\|_{\mathcal L(E,F)}\,\|Z(t-\theta)\|_E,
\]
then applying the Cauchy--Schwarz inequality with respect to the finite measure
$
\|K(t,\theta)\|_{\mathcal L(E,F)}\,|\alpha|(d\theta)
$
gives
\begin{align*}
\int_0^{+\infty}\|\widetilde K(t,\theta)Z(t-\theta)\|_F\,|\alpha|(d\theta)
&\leqslant
\left(\int_0^{+\infty}\|Z(t-\theta)\|_E^2\,\|K(t,\theta)\|_{\mathcal L(E,F)}\,|\alpha|(d\theta)\right)^{1/2}\\
&\qquad\times
\left(\int_0^{+\infty}\|K(t,\theta)\|_{\mathcal L(E,F)}\,|\alpha|(d\theta)\right)^{1/2}\\
&\leqslant
\sqrt{M_0}
\left(\int_0^{+\infty}\|Z(t-\theta)\|_E^2\,\|K(t,\theta)\|_{\mathcal L(E,F)}\,|\alpha|(d\theta)\right)^{1/2}
<\infty,
\end{align*}
where we used Assumption~\ref{ass:lag-kernel} (iii) and \eqref{eq:g-finite}. Hence the Bochner integral
\[
\int_0^{+\infty}\widetilde K(t,\theta)Z(t-\theta)\,|\alpha|(d\theta)
\]
is well defined for almost every $t\in[0,T]$. Thus, we proved (i).
\\
\\
\indent
We next prove (ii). Let $Z_1,Z_2\in L^2(-\infty,T;E)$ satisfy $Z_1=Z_2$ a.e., and set $D:=Z_1-Z_2$. Repeating the previous argument with $D$ in place of $Z$, we obtain
\[
\int_0^T\int_0^{+\infty}\|D(t-\theta)\|_E^2\,\|K(t,\theta)\|_{\mathcal L(E,F)}\,|\alpha|(d\theta)dt=0.
\]
Hence, for almost every $t\in[0,T]$,
\[
\int_0^{+\infty}\|D(t-\theta)\|_E^2\,\|K(t,\theta)\|_{\mathcal L(E,F)}\,|\alpha|(d\theta)=0.
\]
For such a $t$, the same Cauchy--Schwarz estimate yields
\[
\int_0^{+\infty}\|\widetilde K(t,\theta)D(t-\theta)\|_F\,|\alpha|(d\theta)=0,
\]
so that
\[
\int_0^{+\infty}\widetilde K(t,\theta)D(t-\theta)\,|\alpha|(d\theta)=0.
\]
Therefore $\mathcal RZ_1(t)=\mathcal RZ_2(t)$ for almost every $t\in[0,T]$, proving that $\mathcal R$ is well defined on equivalence classes.
\\
\\
\indent
We now prove the $L^2$ bound \eqref{eq:R-bound}. For almost every $t\in[0,T]$, the preceding estimate gives
\begin{align*}
\|\mathcal RZ(t)\|_F^2
&\leqslant
\left(\int_0^{+\infty}\|\widetilde K(t,\theta)Z(t-\theta)\|_F\,|\alpha|(d\theta)\right)^2\\
&\leqslant
\left(\int_0^{+\infty}\|Z(t-\theta)\|_E^2\,\|K(t,\theta)\|_{\mathcal L(E,F)}\,|\alpha|(d\theta)\right)
\left(\int_0^{+\infty}\|K(t,\theta)\|_{\mathcal L(E,F)}\,|\alpha|(d\theta)\right)\\
&\leqslant
M_0\int_0^{+\infty}\|Z(t-\theta)\|_E^2\,\|K(t,\theta)\|_{\mathcal L(E,F)}\,|\alpha|(d\theta).
\end{align*}
Integrating over $t\in[0,T]$ and applying Lemma~\ref{lem:cv-alpha} again, we obtain
\begin{align*}
\int_0^T\|\mathcal RZ(t)\|_F^2dt
&\leqslant
M_0\int_0^T\int_0^{+\infty}\|Z(t-\theta)\|_E^2\,\|K(t,\theta)\|_{\mathcal L(E,F)}\,|\alpha|(d\theta)dt\\
&=
M_0\int_{-\infty}^T\int_0^{+\infty}\|Z(u)\|_E^2\,\|K(u+\theta,\theta)\|_{\mathcal L(E,F)}\mathbf 1_{[0,T]}(u+\theta)\,|\alpha|(d\theta)du\\
&\leqslant
M_0M\int_{-\infty}^T\|Z(u)\|_E^2du.
\end{align*}
This proves \eqref{eq:R-bound}.
\\
\\
\indent
Finally, we verify that $t\mapsto\mathcal RZ(t)$ is strongly measurable. Define
\[
\Phi(t,\theta):=\widetilde K(t,\theta)Z(t-\theta).
\]
We have already shown that $\Phi$ is jointly measurable. Moreover, the pointwise estimate above shows that
\[
a(t):=\int_0^{+\infty}\|\Phi(t,\theta)\|_F\,|\alpha|(d\theta)
\]
belongs to $L^2(0,T)$ and therefore to $L^1(0,T)$. Hence
\[
\int_0^T\int_0^{+\infty}\|\Phi(t,\theta)\|_F\,|\alpha|(d\theta)dt<\infty.
\]
By the Bochner--Fubini theorem, $t\mapsto\int_0^{+\infty}\Phi(t,\theta)\,|\alpha|(d\theta)=\mathcal RZ(t)$ is strongly measurable. Together with \eqref{eq:R-bound}, this proves (iii).
\end{proof}

\begin{remark}
Theorem~\ref{thm:R-bounded} already covers the Dirac case. If $\alpha=\delta_\delta$, then
\[
(\mathcal RZ)(t)=K(t,\delta)Z(t-\delta)
\qquad\text{for almost every }t\in[0,T],
\]
and the map $Z\mapsto \mathcal RZ$ is well defined on $L^2(-\infty,T;E)$ because the ambiguity of pointwise representatives disappears after integration in $t$.
\end{remark}

\begin{theorem}\label{thm:R-adjoint}
Let $\mathcal R:L^2(-\infty,T;E)\to L^2(0,T;F)$ be the bounded operator of Theorem~\ref{thm:R-bounded}. Then its Hilbert adjoint $\mathcal R^*:L^2(0,T;F)\to L^2(-\infty,T;E)$ is given by
\begin{equation}\label{eq:R-star}
(\mathcal R^*Q)(u)
=
\int_0^{+\infty}K(u+\theta,\theta)^*Q(u+\theta)\mathbf 1_{[0,T]}(u+\theta)\,\alpha(d\theta)
\end{equation}
for almost every $u\in(-\infty,T]$. Moreover,
\begin{equation}\label{eq:R-star-bound}
\|\mathcal R^*Q\|_{L^2(-\infty,T;E)}^2
\leqslant M_0M\,\|Q\|_{L^2(0,T;F)}^2,
\qquad Q\in L^2(0,T;F).
\end{equation}
\end{theorem}

\begin{proof}
Fix $Q\in L^2(0,T;F)$ and extend it by $0$ outside $[0,T]$; we keep the same notation. As above, it is enough to work with the positive measure $|\alpha|$ and the kernel $\widetilde K=\jmath K$.

We first show that the right-hand side of \eqref{eq:R-star} is well defined for almost every $u\in(-\infty,T]$. Define
\[
h(u,\theta):=\|Q(u+\theta)\|_F^2\,\|K(u+\theta,\theta)\|_{\mathcal L(E,F)}\mathbf 1_{[0,T]}(u+\theta).
\]
This is a nonnegative measurable function on $(-\infty,T]\times[0,+\infty)$. Applying Lemma~\ref{lem:cv-alpha} with
\[
g(t,\theta):=\|Q(t)\|_F^2\,\|K(t,\theta)\|_{\mathcal L(E,F)},
\]
we obtain
\begin{align*}
\int_{-\infty}^T\int_0^{+\infty}h(u,\theta)\,|\alpha|(d\theta)du
&=
\int_0^T\int_0^{+\infty}\|Q(t)\|_F^2\,\|K(t,\theta)\|_{\mathcal L(E,F)}\,|\alpha|(d\theta)dt\\
&\leqslant
M_0\int_0^T\|Q(t)\|_F^2dt
<\infty.
\end{align*}
Therefore, for almost every $u\in(-\infty,T]$,
\begin{equation}\label{eq:h-finite}
\int_0^{+\infty}\|Q(u+\theta)\|_F^2\,\|K(u+\theta,\theta)\|_{\mathcal L(E,F)}\mathbf 1_{[0,T]}(u+\theta)\,|\alpha|(d\theta)<\infty.
\end{equation}
For such a $u$, applying the Cauchy--Schwarz inequality with respect to the measure
\[
\|K(u+\theta,\theta)\|_{\mathcal L(E,F)}\mathbf 1_{[0,T]}(u+\theta)\,|\alpha|(d\theta)
\]
gives
\begin{align*}
&\int_0^{+\infty}\|\widetilde K(u+\theta,\theta)^*Q(u+\theta)\|_E\mathbf 1_{[0,T]}(u+\theta)\,|\alpha|(d\theta)\\
&\qquad\leqslant
\left(\int_0^{+\infty}\|Q(u+\theta)\|_F^2\,\|K(u+\theta,\theta)\|_{\mathcal L(E,F)}\mathbf 1_{[0,T]}(u+\theta)\,|\alpha|(d\theta)\right)^{1/2}\\
&\qquad\quad\times
\left(\int_0^{+\infty}\|K(u+\theta,\theta)\|_{\mathcal L(E,F)}\mathbf 1_{[0,T]}(u+\theta)\,|\alpha|(d\theta)\right)^{1/2}\\
&\qquad\leqslant
\sqrt{M}
\left(\int_0^{+\infty}\|Q(u+\theta)\|_F^2\,\|K(u+\theta,\theta)\|_{\mathcal L(E,F)}\mathbf 1_{[0,T]}(u+\theta)\,|\alpha|(d\theta)\right)^{1/2}
<\infty.
\end{align*}
Hence the Bochner integral in \eqref{eq:R-star} is well defined for almost every $u$.
\\
\\
\indent
We next prove \eqref{eq:R-star-bound}. For almost every $u\in(-\infty,T]$,
\begin{align*}
\|\mathcal R^*Q(u)\|_E^2
&\leqslant
\left(\int_0^{+\infty}\|\widetilde K(u+\theta,\theta)^*Q(u+\theta)\|_E\mathbf 1_{[0,T]}(u+\theta)\,|\alpha|(d\theta)\right)^2\\
&\leqslant
\left(\int_0^{+\infty}\|Q(u+\theta)\|_F^2\,\|K(u+\theta,\theta)\|_{\mathcal L(E,F)}\mathbf 1_{[0,T]}(u+\theta)\,|\alpha|(d\theta)\right)\\
&\qquad\times
\left(\int_0^{+\infty}\|K(u+\theta,\theta)\|_{\mathcal L(E,F)}\mathbf 1_{[0,T]}(u+\theta)\,|\alpha|(d\theta)\right)\\
&\leqslant
M\int_0^{+\infty}\|Q(u+\theta)\|_F^2\,\|K(u+\theta,\theta)\|_{\mathcal L(E,F)}\mathbf 1_{[0,T]}(u+\theta)\,|\alpha|(d\theta).
\end{align*}
Integrating in $u$ and applying Lemma~\ref{lem:cv-alpha} once again yields
\begin{align*}
\int_{-\infty}^T\|\mathcal R^*Q(u)\|_E^2du
&\leqslant
M\int_{-\infty}^T\int_0^{+\infty}\|Q(u+\theta)\|_F^2\,\|K(u+\theta,\theta)\|_{\mathcal L(E,F)}\mathbf 1_{[0,T]}(u+\theta)\,|\alpha|(d\theta)du\\
&=
M\int_0^T\int_0^{+\infty}\|Q(t)\|_F^2\,\|K(t,\theta)\|_{\mathcal L(E,F)}\,|\alpha|(d\theta)dt\\
&\leqslant
M_0M\int_0^T\|Q(t)\|_F^2dt.
\end{align*}
This proves \eqref{eq:R-star-bound}.
\\
\\
\indent
It remains to verify the adjoint identity. Let $Z\in L^2(-\infty,T;E)$. By the Cauchy--Schwarz inequality on the product space $[0,T]\times[0,+\infty)$ endowed with the measure $dt\otimes|\alpha|(d\theta)$,
\begin{align*}
&\int_0^T\int_0^{+\infty}\bigl|\langle \widetilde K(t,\theta)Z(t-\theta),Q(t)\rangle_F\bigr|\,|\alpha|(d\theta)dt\\
&\qquad\leqslant
\left(\int_0^T\int_0^{+\infty}\|Z(t-\theta)\|_E^2\,\|K(t,\theta)\|_{\mathcal L(E,F)}\,|\alpha|(d\theta)dt\right)^{1/2}\\
&\qquad\quad\times
\left(\int_0^T\int_0^{+\infty}\|Q(t)\|_F^2\,\|K(t,\theta)\|_{\mathcal L(E,F)}\,|\alpha|(d\theta)dt\right)^{1/2}\\
&\qquad\leqslant
\sqrt{M}\,\|Z\|_{L^2(-\infty,T;E)}\,\sqrt{M_0}\,\|Q\|_{L^2(0,T;F)}<\infty.
\end{align*}
Hence Fubini's theorem is applicable, and we compute
\begin{align*}
\int_0^T\langle \mathcal RZ(t),Q(t)\rangle_Fdt
&=
\int_0^T\int_0^{+\infty}\langle \widetilde K(t,\theta)Z(t-\theta),Q(t)\rangle_F\,|\alpha|(d\theta)dt\\
&=
\int_0^T\int_0^{+\infty}\langle Z(t-\theta),\widetilde K(t,\theta)^*Q(t)\rangle_E\,|\alpha|(d\theta)dt.
\end{align*}
For fixed $\theta\geqslant0$, substitute $u=t-\theta$, so that $t=u+\theta$ and $dt=du$. Then
\begin{align*}
&\int_0^T\int_0^{+\infty}\langle Z(t-\theta),\widetilde K(t,\theta)^*Q(t)\rangle_E\,|\alpha|(d\theta)dt\\
&\qquad=
\int_0^{+\infty}\int_{-\theta}^{T-\theta}\langle Z(u),\widetilde K(u+\theta,\theta)^*Q(u+\theta)\rangle_E\,du\,|\alpha|(d\theta)\\
&\qquad=
\int_0^{+\infty}\int_{-\infty}^T\langle Z(u),\widetilde K(u+\theta,\theta)^*Q(u+\theta)\rangle_E\mathbf 1_{[0,T]}(u+\theta)\,du\,|\alpha|(d\theta).
\end{align*}
The absolute integrability proved above allows us to apply Fubini's theorem again and obtain
\begin{align*}
\int_0^T\langle \mathcal RZ(t),Q(t)\rangle_Fdt
&=
\int_{-\infty}^T\left\langle Z(u),\int_0^{+\infty}\widetilde K(u+\theta,\theta)^*Q(u+\theta)\mathbf 1_{[0,T]}(u+\theta)\,|\alpha|(d\theta)\right\rangle_Edu\\
&=
\int_{-\infty}^T\langle Z(u),\mathcal R^*Q(u)\rangle_Edu.
\end{align*}
Since this identity holds for every $Z\in L^2(-\infty,T;E)$, the formula given by \eqref{eq:R-star} is the Hilbert adjoint of $\mathcal R$.
\end{proof}

\begin{remark}
If $\alpha=\delta_\delta$, then \eqref{eq:R-star} reduces to
\[
(\mathcal R^*Q)(u)=K(u+\delta,\delta)^*Q(u+\delta)\mathbf 1_{[0,T]}(u+\delta)
\qquad\text{for almost every }u\leqslant T.
\]
Thus, the point-delay case is included without any additional arguments.
\end{remark}

\subsection{Adapted adjoint.}

Building on the $L^2$-duality results for non-anticipative operators established in Section~\ref{The induced L2-operator and its Hilbert adjoint}, we now extend these identities to the space $\mathcal{M}^2_{\mathcal{F}}$, which is the standard space in stochastic analysis, by incorporating conditional expectation to enforce adaptedness on the adjoint operator.

\begin{proposition}\label{prop:adapted-adjoint}
Assume now that
\[
K:\Omega\times[0,T]\times[0,+\infty)\to\mathcal L(E,F)
\]
satisfies the following conditions:
\begin{enumerate}
\item[(i)] for every $e\in E$, the map $(\omega,t,\theta)\mapsto K(\omega,t,\theta)e$ is jointly measurable;
\item[(ii)] for each $(t,\theta)$, the random operator $K(t,\theta)$ is $\mathcal F_t$-measurable;
\item[(iii)] there exist deterministic constants $M_0,M<\infty$ such that, almost surely,
\[
\sup_{t\in[0,T]}\int_0^{+\infty}\|K(t,\theta)\|_{\mathcal L(E,F)}\,|\alpha|(d\theta)\leqslant M_0
\]
and
\[
\sup_{u\in(-\infty,T]}\int_0^{+\infty}\|K(u+\theta,\theta)\|_{\mathcal L(E,F)}\mathbf 1_{[0,T]}(u+\theta)\,|\alpha|(d\theta)\leqslant M.
\]
\end{enumerate}
Let $Z\in \mathcal{M}^2_{\mathcal F}( -\infty,T;E)$ and $Q\in \mathcal{M}^2_{\mathcal F}(0,T;F)$, and assume in addition that
\[
Z(t)=0,
\qquad t\leqslant0.
\]
Define
\[
(\mathcal RZ)(t):=\int_0^{+\infty}K(t,\theta)Z(t-\theta)\,\alpha(d\theta),
\qquad t\in[0,T],
\]
and
\begin{equation}\label{eq:adapted-adjoint}
(\mathcal A^*Q)(t)
:=
E_t\left[\int_0^{+\infty}K(t+\theta,\theta)^*Q(t+\theta)\mathbf 1_{[0,T]}(t+\theta)\,\alpha(d\theta)\right],
\qquad t\in[0,T].
\end{equation}
where $E_t[\cdot]$ denotes the conditional expectation with respect to $\mathcal{F}_t$. Here $\mathcal A^*Q$ is understood as the optional projection of the square-integrable measurable process $t\mapsto (\mathcal R^*Q)(t)$; equivalently, \eqref{eq:adapted-adjoint} holds for almost every $t\in[0,T]$, almost surely.
Then $\mathcal RZ\in \mathcal{M}^2_{\mathcal F}(0,T;F)$, $\mathcal A^*Q\in \mathcal{M}^2_{\mathcal F}(0,T;E)$, and the following duality identity holds
\begin{equation}\label{eq:adapted-duality}
\mathbb E\bigg[\int_0^T\langle \mathcal RZ(t),Q(t)\rangle_Fdt\bigg]
=
\mathbb E\bigg[\int_0^T\langle Z(t),(\mathcal A^*Q)(t)\rangle_Edt\bigg].
\end{equation}
Moreover,
\begin{equation}\label{eq:adapted-bound}
\mathbb E\bigg[\int_0^T\|(\mathcal A^*Q)(t)\|_E^2dt\bigg]
\leqslant
M_0M\,\mathbb E\bigg[\int_0^T\|Q(t)\|_F^2dt\bigg].
\end{equation}
\end{proposition}

\begin{proof}
Fix $\omega$ outside a null set on which the bounds in (iii) hold. For this $\omega$, the kernel $K(\omega,\cdot,\cdot)$ satisfies the assumptions of Theorems~\ref{thm:R-bounded} and~\ref{thm:R-adjoint} with the same constants $M_0$ and $M$. Applying those theorems pathwise gives, for almost every $\omega$,
\[
\int_0^T\|\mathcal RZ(t,\omega)\|_F^2dt
\leqslant
M_0M\int_{-\infty}^T\|Z(u,\omega)\|_E^2du,
\]
and
\[
\int_{-\infty}^T\|\mathcal R^*Q(u,\omega)\|_E^2du
\leqslant
M_0M\int_0^T\|Q(t,\omega)\|_F^2dt,
\]
where
\[
(\mathcal R^*Q)(u)
:=
\int_0^{+\infty}K(u+\theta,\theta)^*Q(u+\theta)\mathbf 1_{[0,T]}(u+\theta)\,\alpha(d\theta).
\]
Integrating with respect to $P$ yields
\[
\mathbb E\bigg[\int_0^T\|\mathcal RZ(t)\|_F^2dt\bigg]
\leqslant
M_0M\,\mathbb E\bigg[\int_{-\infty}^T\|Z(u)\|_E^2du\bigg],
\]
and
\begin{equation}\label{eq:pathwise-rstar-bound}
\mathbb E\bigg[\int_{-\infty}^T\|\mathcal R^*Q(u)\|_E^2du\bigg]
\leqslant
M_0M\,\mathbb E\bigg[\int_0^T\|Q(t)\|_F^2dt\bigg].
\end{equation}
Since $K(t,\theta)$ and $Z(t-\theta)$ are both $\mathcal F_t$-measurable, the integrand in $\mathcal RZ(t)$ is $\mathcal F_t$-measurable for each fixed $t$. Therefore $\mathcal RZ$ is adapted, hence $\mathcal RZ\in \mathcal{M}^2_{\mathcal F}(0,T;F)$.
\\
\\
\indent
Moreover, the process $t\mapsto \mathcal R^*Q(t)$ is jointly measurable and square integrable by \eqref{eq:pathwise-rstar-bound}. Hence its optional projection exists, and by definition this optional projection is $\mathcal A^*Q$. In particular, $\mathcal A^*Q$ is adapted and jointly measurable.
\\
\noindent
Next, by Jensen's inequality for conditional expectation and \eqref{eq:pathwise-rstar-bound},
\begin{align*}
\mathbb E\bigg[\int_0^T\|(\mathcal A^*Q)(t)\|_E^2dt\bigg]
&=
\mathbb E\bigg[\int_0^T\left\|E_t\bigl[(\mathcal R^*Q)(t)\bigr]\right\|_E^2dt\bigg]\\
&\leqslant
\mathbb E\bigg[\int_0^T E_t\bigl[\|\mathcal R^*Q(t)\|_E^2\bigr]dt\bigg]\\
&=
\mathbb E \bigg[\int_0^T\|\mathcal R^*Q(t)\|_E^2dt\bigg]\\
&\leqslant
\mathbb E\bigg[\int_{-\infty}^T\|\mathcal R^*Q(t)\|_E^2dt\bigg]\\
&\leqslant
M_0M\,\mathbb E\bigg[\int_0^T\|Q(t)\|_F^2dt\bigg].
\end{align*}
Thus \eqref{eq:adapted-bound} holds, and in particular $\mathcal A^*Q\in \mathcal{M}^2_{\mathcal F}(0,T;E)$.
\\
\indent

We now prove the duality identity. By the path-wise adjoint formula from Theorem~\ref{thm:R-adjoint} and the assumption \(Z(t)=0\) for \(t\leqslant 0\), we have for almost every \(\omega\),
\[
\int_0^T \langle\mathcal RZ(t,\omega),Q(t,\omega)\rangle_F\,dt
=
\int_0^T \langle Z(t,\omega),\mathcal R^*Q(t,\omega)\rangle_E\,dt.
\]
Moreover, by Cauchy--Schwarz inequality and \eqref{eq:pathwise-rstar-bound},
\[
\mathbb{E}\bigg[\int_0^T \bigl|\langle Z(t),\mathcal R^*Q(t)\rangle_E\bigr|\,dt\bigg]
\leqslant
\left(\mathbb{E}\bigg[\int_0^T \|Z(t)\|_E^2\,dt\bigg]\right)^{1/2}
\left(\mathbb{E}\bigg[\int_0^T \|\mathcal R^*Q(t)\|_E^2\,dt\bigg]\right)^{1/2}
<\infty,
\]
and similarly
\[
\mathbb{E}\bigg[\int_0^T \bigl|\langle \mathcal RZ(t),Q(t)\rangle_F\bigr|\,dt\bigg]<\infty.
\]
Hence we may take expectations in the above path-wise identity and obtain
\[
\mathbb{E}\bigg[\int_0^T \langle\mathcal RZ(t),Q(t)\rangle_F\,dt\bigg]
=
\mathbb{E}\bigg[\int_0^T \langle Z(t),\mathcal R^*Q(t)\rangle_E\,dt\bigg].
\]
Since \(Z(t)\) is \(\mathcal{F}_t\)-measurable, for almost every \(t\in[0,T]\),
\[
\mathbb{E}\bigl[\langle Z(t),\mathcal R^*Q(t)\rangle_E\bigr]
=
\mathbb{E}\bigl[\langle Z(t),E_t[\mathcal R^*Q(t)]\rangle_E\bigr]
=
\mathbb{E}\bigl[\langle Z(t),(\mathcal A^*Q)(t)\rangle_E\bigr].
\]
Integrating over \(t\in[0,T]\) yields \eqref{eq:adapted-duality}.
\end{proof}

\begin{corollary}\label{cor:section3-specialization}
Assume that the control space is a separable Hilbert space $U$, and set $E=F=U$. Let
\[
K(t,\theta)=\phi(t-\theta,t)I_U,
\qquad t\in[0,T],\ \theta\geqslant0.
\]
Then
\[
(\mathcal Rv)(t)=\int_0^{+\infty}\phi(t-\theta,t)v(t-\theta)\,\alpha(d\theta)=:v_d(t),
\qquad t\in[0,T].
\]
Moreover, under Assumption~\ref{ass: delay measure}, the constants in Assumption~\ref{ass:lag-kernel} satisfy
\[
M_0\leqslant C_\phi |\alpha|([0,+\infty)),
\qquad
M\leqslant C_\phi |\alpha|([0,+\infty)).
\]
Hence Proposition~\ref{prop:adapted-adjoint} applies, and the adapted adjoint is
\[
(\mathcal A^*Q)(t)
=
E_t\left[\int_0^{+\infty}\phi(t,t+\theta)Q(t+\theta)\,\alpha(d\theta)\right],
\qquad t\in[0,T].
\]
\end{corollary}

\begin{proof}
The identity for $\mathcal Rv$ follows immediately from the definition of $K$. Since
\[
K(t+\theta,\theta)=\phi(t,t+\theta)I_U,
\]
formula \eqref{eq:adapted-adjoint} gives
\[
(\mathcal A^*Q)(t)
=
E_t\left[\int_0^{+\infty}\phi(t,t+\theta)I_U^*Q(t+\theta)\mathbf 1_{[0,T]}(t+\theta)\,\alpha(d\theta)\right].
\]
Because $t\in[0,T]$ and $\theta\geqslant0$, the indicator $\mathbf 1_{[0,T]}(t+\theta)$ simply restricts the integral to those $\theta$ for which $t+\theta\leqslant T$. Noticing $I_U^*=I_U$ and
\[
K(t+\theta,\theta)^*
=
\bigl[\phi(t,t+\theta)I_U\bigr]^*
=
\phi(t,t+\theta)I_U^*,
\]
the adapted adjoint reduces to
\[
(\mathcal A^*Q)(t)
=
E_t\left[\int_0^{+\infty}\phi(t,t+\theta)Q(t+\theta)\,\alpha(d\theta)\right].
\]
Taking $Q=H^u_{v}$ coincides with the term appearing in $(\ref{condition neccesary SMP})$.
\end{proof}

	\section{Optimal Control Problem.}\label{section SMP}
		
		In this section, we study the stochastic maximum principle for a stochastic delayed control system described by infinitely delayed SEE. Based on the duality property introduced in Section \ref{chapter Non-anticipative path derivatives and infinite-window dual operators in Hilbert spaces} and the variational inequality, we give the necessary conditions of the SMP in Section \ref{chapter necessary SMP}. Then we present the sufficient optimality conditions in Section \ref{chapter sufficient SMP} with additional convexity assumptions.
		\subsection{Formulation of the optimal control problem.}
		Let $U$ be the control domain with scalar product $\left\langle \cdot, \cdot \right\rangle_{U}$ and norm $\|\cdot\|_U$. Suppose $U$ is a nonempty convex subset of a real separable Hilbert space $\mathcal{U}$ which is identified with its dual space.  Consider the following controlled infinitely delayed stochastic evolution equation on the Hilbert space $H$:
		\begin{equation}\label{ISDDE with control}
			\left\{
			\begin{aligned}
				dX^v(t) &= b(t,X^v_t, v_d(t))dt + \sigma(t,X^v_t,v_d(t))dW(t), \qquad t\in[0,T],\\
				X^v(t) &= \gamma(t), \quad v(t) = \varphi(t), \qquad t\in (-\infty, 0],
			\end{aligned}
			\right.
		\end{equation}
        
		Here, the initial datum $\gamma(\cdot)\in \mathcal{M}_\mathcal{F}^2(-\infty, 0 ; H)$ satisfies $\gamma_0 \in L^2(\mathcal{F}_0;C_\lambda((-\infty,0];H))$ and $\varphi(\cdot)\in \mathcal{M}_\mathcal{F}^2(-\infty, 0 ; U)$. The random functions
		\[
		b:[0,T]\times \Omega\times C_\lambda((-\infty,T];H)\times \mathcal{U}\to H,\quad 
		\sigma:[0,T]\times \Omega\times C_\lambda((-\infty,T];H)\times\mathcal{U} \to \mathcal{L}_2(\mathbb{R}^m,H)
		\]
		satisfy that $b(\cdot,\cdot,x,v)$, $\sigma(\cdot,\cdot,x,v)$ are progressively measurable for each $(x,v)\in C_\lambda((-\infty,T];H)\times \mathcal{U}$ .
		
		To account for the memory effect in the control input in (\ref{ISDDE with control}), we define the delayed control $v_d(\cdot)$ via an integral with respect to a $\sigma$-finite measure:
		\begin{equation}\nonumber
			v_d(t) := \int_{0}^{+\infty}\phi(t-\theta,t)v(t-\theta)\alpha(d\theta) , \quad t \in [0, T],
		\end{equation}
where the scalar kernel $\phi(\cdot, \cdot)$ is bounded and $\alpha$ is a $\sigma$-finite measure on $[0,+\infty)$. Denote by $\left|\alpha\right|$ the total variation of measure $\alpha$ and assume:
		\begin{assumption}\label{ass: delay measure}
		     There exists constants $C_{\phi}>0$ and $C_{\alpha}>0$ such that 
		\[
        |\phi(\cdot,\cdot)| \leqslant C_{\phi}, 
		\quad\int_{0}^{+\infty}\left|\alpha\right|(d\theta) \leqslant C_{\alpha}.
        \]
		\end{assumption}
        
		Define the set of admissible controls as 
		\[
		\mathcal{U}_{a d}=\left\{v(\cdot):(-\infty, T] \times\Omega\rightarrow U \mid v(\cdot)\in \mathcal{M}_{\mathcal{F}}^2\left(0, T ; U\right)~ \text{and}~ v(t)=\varphi(t) \text { for } t \leqslant 0\right\}.
		\]
		
The objective is to minimize the following performance functional over $\mathcal{U}_{a d}$
	\begin{equation}\label{cost functional}
	\begin{aligned}
		J(v(\cdot))=\mathbb{E} \left[\int_{0}^{T} l\left(t, X_t^{v}, v_d(t)\right) d t\right]+ \mathbb{E}\big[h(X^v(T))\big],
	\end{aligned}
\end{equation}
where  $l:[0,T]\times \Omega \times C_\lambda((-\infty,T];H)\times \mathcal{U} \rightarrow \mathbb{R}$ and $h:\Omega \times H \rightarrow \mathbb{R}$ are given measurable functions such that $l(\cdot, \cdot, x, v)$ is progressively measurable for each $(x,v)\in C_\lambda((-\infty,T];H)\times \mathcal{U}$ and $h(\cdot, x_T)$ is $\mathcal{F}_T$ measurable for each $x_T\in H$. 

We introduce the following assumptions:
		 \begin{assumption}\label{ass: Integrability} (Integrability) 
The processes \(\{b(t,0,0)\}_{t\in[0,T]}\) and \(\{\sigma(t,0,0)\}_{t\in[0,T]}\) satisfy
\[
\left\{b\left(t, 0,0\right)\right\}_{t \in[0, T]} \in \mathcal{M}_{\mathcal{F}}^2(0, T ; H),\quad \left\{\sigma\left(t, 0,0\right)\right\}_{t \in[0, T]} \in \mathcal{M}_{\mathcal{F}}^2\left(0, T ; \mathcal{L}_2(\mathbb{R}^m,H) \right)
\]
Moreover, for all $v(\cdot)\in \mathcal{U}_{a d}$,
\[
\mathbb{E}\left[\int_0^T |\,l(t,X_t^v,v_d(t))\,|\,dt\right]<\infty .
\] 
\end{assumption}

	\begin{assumption}\label{ass: differentiability} (Differentiability)	In this assumption, we write $\mathcal C_\lambda:=C_\lambda((-\infty,T];H)$ for brevity.
	    \begin{enumerate}
\item[(i)] For almost all $(t, \omega)\in [0,T]\times \Omega$, $b$ and $\sigma$ are continuously Fr\'echet differentiable with respect to $(x, v)\in \mathcal C_\lambda \times \mathcal{U}$. Their  Fr\'echet derivatives satisfy
\begin{align*}
    b_x(t,\omega,x,v)&\in \mathcal L(\mathcal C_\lambda,H),
&&\sigma_x(t,\omega,x,v)\in \mathcal L(\mathcal C_\lambda, \mathcal L_2(\mathbb{R}^m,H)),\\
b_v(t,\omega,x,v)&\in \mathcal L(\mathcal U,H),
&&\sigma_v(t,\omega,x,v)\in \mathcal L(\mathcal U,\mathcal L_2(\mathbb{R}^m,H)).
\end{align*}
Moreover, there exists a constant \(N_1>0\) such that, for almost all \((t,\omega)\) and all \((x,v)\in\mathcal C_\lambda\times \mathcal U\),
\begin{align*}
    &\|b_x(t,\omega,x,v)\|_{\mathcal L(\mathcal C_\lambda,H)}
+\|\sigma_x(t,\omega,x,v)\|_{\mathcal L(\mathcal C_\lambda,\mathcal L_2(\mathbb{R}^m,H))}\\
&+\|b_v(t,\omega,x,v)\|_{\mathcal L(\mathcal U,H)}
+\|\sigma_v(t,\omega,x,v)\|_{\mathcal L(\mathcal U,\mathcal L_2(\mathbb{R}^m,H))}
\leqslant N_1.
\end{align*}

\item[(ii)] For almost all $(t, \omega)\in [0,T]\times \Omega$, $l$ is  continuously Fr\'echet differentiable with respect to $(x, v)\in C_{\lambda}\times \mathcal{U}$ and $h$ is continuously Fr\'echet differentiable. 
Under the identification \(U^*\simeq U\), their Fr\'echet derivatives satisfy
\[
l_x(t,\omega,x,v)\in \mathcal L(C_\lambda, \mathbb{R}) ,\qquad
l_v(t,\omega,x,v)\in U,\qquad
h_{x_T}(x_T)\in H.
\]
Moreover, there exists a constant \(N_2>0\) such that, for almost all \((t,\omega)\) and all \((x,v,x_T)\in \mathcal C_\lambda\times \mathcal U\times H\),
\[
|l(t,\omega,x,v)|\leqslant N_2\big(1+\|x\|_{\lambda,T}^2+\|v\|_U^2\big),
\]
\[
\|l_x(t,\omega,x,v)\|_{\mathcal L(C_\lambda, \mathbb{R})}+\|l_v(t,\omega,x,v)\|_U
\leqslant N_2\big(1+\|x\|_{\lambda,T}+\|v\|_U\big),
\]
\[
|h(x_T)|+\|h_{x_T}(x_T)\|_H
\leqslant N_2\big(1+\|x_T\|_H\big).
\]

    \item[(iii)]
The functions $b, \sigma$ and $l$ are non-anticipatively differentiable with respect to $x\in \mathcal C_\lambda$ in the sense of Definition~\ref{def: non-anticipatively differentiable}, with path derivatives $\partial
 	_{x} b(t,x,v)$, $\partial _{x} \sigma (t,x,v)$ and $\partial
 	_{x} l(t,x,v)$.
Moreover, for each fixed \((x,v)\in \mathcal C_\lambda\times \mathcal{U}\), the corresponding kernels satisfy Assumption~\ref{ass:lag-kernel} and the measurability requirements of Proposition~\ref{prop:adapted-adjoint}, with common deterministic bounds \(M_0\) and \(M\), uniformly in \((\omega,x,v)\).
\end{enumerate}
	\end{assumption}

    \begin{remark}\label{connection of 2 derivatives}
For $\Phi=b,\sigma,l$, we use $\Phi_x(t,x,v)$ to denote the Fr\'echet derivative on the phase space $C_{\lambda}(( -\infty, T];H)$. Since the coefficients are non-anticipative, we use $\partial_x\Phi(t,x,v)$ to denote the corresponding non-anticipative path derivative, which acts only on the past-restricted subspace $C_{\lambda,t}(( -\infty, T];H) \subsetneq C_{\lambda}(( -\infty, T];H)$.

The Fr\'echet derivative considers all perturbations on $C_{\lambda}(( -\infty, T];H)$, while the non-anticipative derivative only concerns the perturbations supported in $(-\infty,t]$. In Section~\ref{section SMP}, we use $C_{\lambda}(( -\infty, T];H)$ as the ambient phase space and use $C_{\lambda,t}(( -\infty, T];H)$ whenever the non-anticipative structure of time $t$ must be made explicit.
\end{remark}

    \begin{remark}
        	By Theorem~\ref{thm:ISFDE-wellposed}, it follows from Assumption~\ref{ass: Integrability} and Assumption~\ref{ass: differentiability}  (i) that the controlled SEE (\ref{ISDDE with control}) admits a unique solution $X^v(\cdot)\in\mathcal{S}_\mathcal{F}^2(-\infty, T ; H) $ in the sense of Definition~\ref{definition X} for all $v(\cdot)\in \mathcal{U}_{a d}$.  Moreover, Assumption~\ref{ass: differentiability}  (ii) implies that $\left| J(v(\cdot))\right|< +\infty$.
    \end{remark}

	\subsection{Necessary Conditions.}\label{chapter necessary SMP}
	Let $u(\cdot)\in \mathcal{U}_{a d}$ be the optimal control of our problem and denote by $X^u(\cdot)$ the optimal trajectory. Taking an arbitrary $v(\cdot)\in \mathcal{U}_{a d}$, since $U$ is convex, we know that $u^{\varepsilon}(\cdot):= u(\cdot) + \varepsilon (v(\cdot)-u(\cdot)) \in \mathcal{U}_{a d}$ for $0 \leqslant \varepsilon \leqslant 1$. Let $X^{\varepsilon}(\cdot)$ be the corresponding trajectory of system (\ref{ISDDE with control}) with perturbed control $v^{\varepsilon}(\cdot)$.

For simplicity, we denote:
\begin{align*}
    \hat{v}(\cdot):= v(\cdot)-u(\cdot),\quad
    \hat{v}_d(\cdot):= v_d(\cdot)-u_d(\cdot),
\end{align*}
and introduce the following abbreviations for the functions $\Phi =  b,\sigma,l$ and their derivatives evaluated at the optimal pair $\left(u(\cdot), X^u(\cdot)\right)$:
\begin{align*}
	\Phi^u(t)&:=\Phi (t, X^{u}_{t},u_{d}(t)), \\
		\partial ^u_{x}\Phi (t)&:=\partial _{x}\Phi (t,X^u_{t},u_{d}(t)), \\
        \Phi^u_v(t)&:= \Phi_v (t,X^u_{t},u_{d}(t)).
\end{align*}  

To derive the first-order necessary condition in terms of small $\varepsilon$, we consider the following linear infinitely delayed SEE, which we call the "variational equation":
\begin{equation}\label{eq: variational equation}
    \left\{
			\begin{aligned}
				d\hat X(t) &= \Big[\partial^u_x b(t) (\hat X_t) + b^u_v(t)\hat{v}_d(t)\Big]dt + \Big[\partial^u_x \sigma(t) (\hat X_t) + \sigma^u_v(t)\hat{v}_d(t)\Big] dW(t), \quad t\in[0,T],\\
				\hat X(t) &= 0, \qquad t\in (-\infty, 0].
			\end{aligned}
			\right.
\end{equation}

\begin{remark}
    For the well-posedness of the equation (\ref{eq: variational equation}), it suffices to show that Assumptions~\ref{ass:ISFDE-Lip}-\ref{ass:ISFDE-int} hold for the variational equation (\ref{eq: variational equation}). For each fixed perturbation $\hat v(\cdot)$, define
\[
\bar b(t,x_t):=\partial_x^u b(t)(x_t)+b_v^u(t)\hat v_d(t),
\qquad
\bar\sigma(t,x_t):=\partial_x^u \sigma(t)(x_t)+\sigma_v^u(t)\hat v_d(t).
\]
By Assumption~\ref{ass: differentiability} (i) and the kernel bounds in Assumption~\ref{ass: differentiability} (iii), there exists a constant $C>0$ such that, for all $x_t,x_t'\in C_{\lambda,t}(( -\infty, T];H)$,
\[
\|\bar b(t,x_t)-\bar b(t,x_t')\|_H +\|\bar\sigma(t,x_t)-\bar\sigma(t,x_t')\|_{\mathcal L_2}
\leqslant C\|x_t-x_t'\|_{\lambda,T}.
\]
Moreover, $\bar b(\cdot,0)\in \mathcal{M}_{\mathcal F}^2(0,T;H)$ and $\bar\sigma(\cdot,0)\in \mathcal M_{\mathcal F}^2(0,T;\mathcal{L}_2(\mathbb R^m,H))$. Hence Assumptions~\ref{ass:ISFDE-Lip}-\ref{ass:ISFDE-int} hold for the variational equation (\ref{eq: variational equation}), and Theorem~\ref{thm:ISFDE-wellposed} implies that (\ref{eq: variational equation}) admits a unique solution.
\end{remark}

For each $t\in [0,T]$, we set
\[
\tilde{X}^{\varepsilon}(t):= \frac{X^{\varepsilon}(t) - X^u(t)}{\varepsilon} - \hat{X}(t).
\]

Then we have the following convergence result.
	\begin{lemma}\label{variational estimate}
			Suppose Assumption~\ref{ass: Integrability}-\ref{ass: differentiability} hold, then we have 
			\[
            \lim_{\varepsilon \rightarrow 0}\mathbb{E}\left[\sup_{0 \leqslant t \leqslant T}\left\|\tilde{X}^{\varepsilon}(t)\right\|_H^2\right]=0.
            \]
		\end{lemma}
		\begin{proof} By the definition of $\tilde{X}^{\varepsilon}(t)$, we know that it satisfies
			\[
            \left\{\begin{aligned}
				d\tilde{X}^{\varepsilon}(t)=&\bigg[\frac{b(t, X^{\varepsilon}(t), v_d^{\varepsilon}(t))-b^u(t)}{\varepsilon}-\partial^u_xb(t)(\hat{X}_t)-b^u_v(t)\hat{v}_{d}(t)\bigg]dt\\
				& + \bigg[\frac{\sigma(t, X^{\varepsilon}(t), v_d^{\varepsilon}(t))-\sigma^u(t)}{\varepsilon}-\partial^u_x \sigma(t)(\hat{X}_t)-\sigma^u_v(t)\hat{v}_{d}(t)\bigg] d W(t), \quad  t \in[0, T]; \\
				\tilde{X}^{\varepsilon}(t) = &0, \quad t\in(-\infty, 0].
			\end{aligned}\right.
            \]
            It follows from the the first-order Taylor’s expansion of $b, \sigma$ that the above dynamic of $\tilde{X}^{\varepsilon}$ can be written as
			\[
            \left\{\begin{aligned}
				d\tilde{X}^{\varepsilon}(t)=&\Big\{\partial_x^{\varepsilon} b(t)(\tilde{X}^{\varepsilon}_t) + \big[\partial_x^{\varepsilon} b(t)-\partial^u_xb(t)\big](\hat{X}_t)+\big[ b^{\varepsilon}_v(t)-b^u_v(t)\big]\hat{v}_{d}(t)\Big\}dt\\
				& + \Big\{\partial_x^{\varepsilon} \sigma(t)(\tilde{X}^{\varepsilon}_t) + \big[\partial_x^{\varepsilon} \sigma(t)-\partial^u_x\sigma(t)\big](\hat{X}_t)+\big[ \sigma^{\varepsilon}_v(t)-\sigma^u_v(t)\big]\hat{v}_{d}(t)\Big\} d W(t), \quad  t \in[0, T]; \\
				\tilde{X}^{\varepsilon}(t) = &0, \quad t\in(-\infty, 0],
			\end{aligned}\right. 
            \]
            where we denote for $\Phi = b, \sigma$,
            \begin{align*}
                \partial_x^{\varepsilon} \Phi(t)&:=  \int_{0}^{1}\partial_x\Phi\left(t,X^u_t+\lambda(X^{\varepsilon}_t-X^u_t),u_d(t)\right)d\lambda,   \\ 
                \Phi_v^{\varepsilon}&:=\int_{0}^{1}\Phi_v\left(t,X^u_t,u_d(t) +\lambda(u_d^{\varepsilon}(t)-u_d(t))\right)d\lambda.
            \end{align*}
        Applying the \textit{a priori} estimate (\ref{a priori est SEE}) yields 
        \begin{align*}
            \mathbb{E}\left[\sup_{0 \leqslant t \leqslant T}\left\|\tilde{X}^{\varepsilon}(t)\right\|_H^2\right]\leqslant
            C\bigg\{
	&\mathbb{E}\Big[\int_0^T \left\|\big[\partial_x^{\varepsilon} b(t)-\partial^u_xb(t)\big](\hat{X}_t)+\big[ b^{\varepsilon}_v(t)-b^u_v(t)\big]\hat{v}_{d}(t)\right\|_H^2dt\Big]\\
	&+\mathbb{E}\Big[\int_0^T \left\|\big[\partial_x^{\varepsilon} \sigma(t)-\partial^u_x\sigma(t)\big](\hat{X}_t)+\big[ \sigma^{\varepsilon}_v(t)-\sigma^u_v(t)\big]\hat{v}_{d}(t)\right\|_{\mathcal{L}_2}^2\,dt\Big]
	\bigg\} \rightarrow0,
        \end{align*}
       as $\varepsilon \rightarrow0$, where the convergence holds by the continuous differentiability of $b, \sigma$ and the dominated convergence theorem.
		\end{proof}
\\

    Since $u(\cdot)$ is the optimal control of the delayed system (\ref{ISDDE with control}), i.e. $\varepsilon^{-1}\left[J\left(u^{\varepsilon}(\cdot)\right)-J\left(u(\cdot)\right)\right] \geqslant 0 $, the  following variational inequality follows directly by using 
	Lemma \ref{variational estimate} and Taylor's expansion of $l, \gamma$.

    \begin{lemma}
Suppose that Assumptions~\ref{ass: Integrability}-\ref{ass: differentiability} hold. Then the variational inequality holds:
\begin{equation}\label{eq variational inequality}
    \mathbb E\left[\int_0^T\Big(\partial_x^u l(t)(\hat X_t)+\langle l_v^u(t),\hat v_d(t)\rangle_U\Big)\,dt
+\langle h_{x_T}(X^u(T)),\hat X(T)\rangle_H\right]\geqslant 0.
\end{equation}
\end{lemma}

To derive the maximum principle, we introduce the adjoint equation associated with the variational equation~(\ref{eq: variational equation}). 
\begin{equation}\label{adjoint equation}
\left\{
\begin{aligned}
-dp(t)&=E_t\Big[\partial_x^{\ast,u}b(t)(p_{t+})+\partial_x^{\ast,u}\sigma(t)(q_{t+})+\partial_x^{\ast,u}l(t)(1)\Big]dt-q(t)dW(t),\quad t\in[0,T],\\
p(T)&=h_{x_T}(X^u(T)),\qquad
p(t)=0, \quad t\in(T,\infty),   \qquad q(t)=0,\quad t\in[T,\infty).
\end{aligned}
\right.
\end{equation}

Here, for $\Phi=b,\sigma,l$, $\partial_x^{\ast,u}\Phi(t)$ denotes the path-wise adjoint operator of $\partial_x^u\Phi(t)$, and let $E_t[\partial_x^{\ast,u}\Phi(t)(\cdot)]$ denote its adapted version. In addition, $\partial_x^{\ast,u}\Phi(t)(1)$ means that the operator $\partial_x^{\ast,u}\Phi(t)$ acts on the constant scalar process $1$.

\begin{remark}
    The adjoint equation (\ref{adjoint equation}) forms a special case of the IABSEE (\ref{ABSDE without control}), since the adapted generator in  (\ref{adjoint equation}) depends on the whole future value of $(p_{t+}, q_{t+})$. 

    By Proposition~\ref{prop:adapted-adjoint}, the adapted adjoint operators are bounded operators on $\mathcal{M}^2_{\mathcal{F}}(0,T; \cdot)$ with estimate (\ref{eq:adapted-bound}). Hence,  Assumption~\ref{ass:IABSDE-Lip} holds for equation (\ref{adjoint equation}), and by Theorem~\ref{thm:IABSDE-wellposed}, equation (\ref{adjoint equation}) has a unique solution $(p(\cdot), q(\cdot))\in\mathcal{S}_\mathcal{F}^2\left(0, +\infty; H\right) \times \mathcal{M}_\mathcal{F}^{2,\beta}\left(0, +\infty; \mathcal{L}_2(\mathbb{R}^m,H)\right)$.

\end{remark}

		Define the Hamiltonian function $H: [0,	T] \times \Omega\times  C_\lambda((-\infty,T];H)\times \mathcal{U} \times H \times \mathcal{L}_2(\mathbb{R}^m,H)\rightarrow  \mathbb{R}$: 
		\[
        		H\left(t, x, v, p, q\right):= \left\langle b(t, x, v), p \right\rangle_H  + \left\langle \sigma(t, x, v), q \right\rangle_2 +l\left(t, x, v\right),
        \]
with associated shorthand:
\begin{align*}
	H^u(t)&:=H (t, X^{u}_{t},u_{d}(t), p(t), q(t)), \\
		\partial ^u_{x}H (t)&:=\partial _{x}H (t,X^u_{t},u_{d}(t),p(t), q(t)), \\
        H^u_v(t)&:= H_v (t,X^u_{t},u_{d}(t),p(t), q(t)),
\end{align*}  
where $(p(\cdot), q(\cdot))$ solve the equation (\ref{adjoint equation}).
        \\
        
		Now we can give the necessary condition for the maximum principle.
		\begin{theorem}\label{necessary SMP}
			(Necessary conditions of optimality) Suppose Assumptions~\ref{ass: Integrability}-\ref{ass: differentiability} hold. Let $u(\cdot)$ be an optimal control  of the stochastic control system with delay (\ref{ISDDE with control})-(\ref{cost functional}),  $X^u(\cdot)$ be the corresponding optimal state trajectory and $(p(\cdot), q(\cdot))$ be the solution of the adjoint equation (\ref{adjoint equation}). Then
			\begin{equation}\label{condition neccesary SMP}
				\begin{aligned}
					\left\langle E_t\left[ \int_{0}^{+\infty}\phi(t,t+\theta)H^u_{v}(t+\theta)\alpha(d\theta)\right],v-u(t)\right\rangle_U  \geqslant 0,
				\end{aligned}
			\end{equation}
			a.e., a.s. for all $v\in U$.
		\end{theorem}
		
		\begin{proof}
			Applying It\^o's formula (\ref{ito on 2 product}) to $\left\langle p(t), \hat{X}(t)\right\rangle_H$ on $[0,T]$ and taking expectation yields that
            \begin{align}
                  \mathbb{E}&\left[\left\langle h_{x_{T}}(X^u(T)),\hat{X}(T) 
                \right\rangle_H \right] \label{last term in SMP} \\
                &= \mathbb{E}\left[\int_0^T \left\langle \partial_x^ub(t)(\hat{X}_t) + b_v^u(t) \hat{v}_d(t), p(t) 
                \right\rangle_H dt\right] \nonumber\\
                &\quad - \mathbb{E}\left[\int_0^T \left\langle E_t\Big[\partial_x^{*,u}b(t)(p_{t+}) + \partial_x^{*,u}\sigma(t)(q_{t+}) + \partial_x^{*,u}l(t)(1)\Big], \hat{X}(t)
                \right\rangle_H dt\right] \nonumber\\
                &\quad + \mathbb{E}\left[\int_0^T \left\langle \partial_x^u\sigma(t)(\hat{X}_t) + \sigma_v^u(t) \hat{v}_d(t), q(t)
                \right\rangle_2 dt\right] ,\nonumber
            \end{align}
Since $\hat{X}(t) = 0$ for $t\leqslant 0$, then by (\ref{eq:adapted-duality}), we can obtain the following duality relationships
\begin{align*}
\mathbb{E}\left[\int_0^T \left\langle \partial_x^ub(t)(\hat{X}_t), p(t) \right\rangle_H dt\right]& = \mathbb{E}\left[\int_0^T \left\langle E_t\Big[\partial_x^{*,u}b(t)(p_{t+})\Big], \hat{X}(t)
                \right\rangle_H dt\right],\\
       \mathbb{E}\left[\int_0^T \left\langle \partial_x^u\sigma(t)(\hat{X}_t), q(t) \right\rangle_H dt\right]& = \mathbb{E}\left[\int_0^T \left\langle E_t\Big[\partial_x^{*,u}\sigma(t)(q_{t+})\Big], \hat{X}(t)
                \right\rangle_H dt\right].
\end{align*}
Taking $Q \equiv 1$ in the scalar-valued case of Proposition~\ref{prop:adapted-adjoint} yields
\[
\mathbb{E}\left[\int_0^T \partial_x^ul(t)(\hat{X}_t) dt\right]= \mathbb{E}\left[\int_0^T \left\langle E_t\Big[\partial_x^{*,u}l(t)(1)\Big], \hat{X}(t)
                \right\rangle_H dt\right].
\]
\noindent
For $t \in [0, T]$, define
\[
Q_b(t) := (b_v^u(t))^* p(t), \qquad Q_\sigma(t) := (\sigma_v^u(t))^* q(t), \qquad Q_l(t) := l_v^u(t),
\]
and extend $Q_b, Q_\sigma, Q_l$ by $0$ on $(T, \infty)$. Then
\[
H_v^u(t) = Q_b(t) + Q_\sigma(t) + Q_l(t), \qquad t \geqslant 0.
\]
Since $\hat{v}(t) = 0$ for $t \leqslant 0$, applying Corollary~\ref{cor:section3-specialization} with $Q = Q_b, Q_\sigma, Q_l$ yields
\begin{align*}
    \mathbb{E}&\left[\int_0^T \left\langle b_v^u(t) \hat{v}_d(t), p(t) \right\rangle_H dt\right]
    = \mathbb{E}\bigg[\int_0^T \langle \hat v_d(t),Q_b(t)\rangle_U dt\bigg]\\
    &= \mathbb{E}\left[\int_0^T \left\langle E_t\Big[ \int_{0}^{+\infty}\phi(t,t+\theta)Q_b(t+\theta)\alpha(d\theta), \hat{v}(t) \right\rangle_U dt\right],\\
    \mathbb{E}&\left[\int_0^T \left\langle \sigma_v^u(t) \hat{v}_d(t), q(t) \right\rangle_2 dt\right]
    = \mathbb{E}\bigg[\int_0^T \langle \hat v_d(t),Q_{\sigma}(t)\rangle_U dt\bigg]\\
    &= \mathbb{E}\left[\int_0^T \left\langle E_t\Big[ \int_{0}^{+\infty}\phi(t,t+\theta)Q_{\sigma}(t+\theta)\alpha(d\theta), \hat{v}(t) \right\rangle_U dt\right],
\end{align*}
and
\begin{align*}
    \mathbb{E}&\left[\int_0^T \left\langle l_v^u(t) , \hat{v}_d(t) \right\rangle_U dt\right]
    = \mathbb{E}\bigg[\int_0^T \langle \hat v_d(t),Q_{l}(t)\rangle_U dt\bigg]\\
    &= \mathbb{E}\left[\int_0^T \left\langle E_t\Big[ \int_{0}^{+\infty}\phi(t,t+\theta)Q_{l}(t+\theta)\alpha(d\theta), \hat{v}(t) \right\rangle_U dt\right].
\end{align*}
\noindent
Substituting (\ref{last term in SMP}) and the above duality relations into the variational inequality (\ref{eq variational inequality}), we have 
\begin{align*}
    &\mathbb{E}\left[\int_0^T \left\langle E_t\Big[ \int_{0}^{+\infty}\phi(t,t+\theta)Q_{b}(t+\theta)\alpha(d\theta), \hat{v}(t) \right\rangle_U dt\right] \\
    &+ \mathbb{E}\left[\int_0^T \left\langle E_t\Big[ \int_{0}^{+\infty}\phi(t,t+\theta)Q_{\sigma}(t+\theta)\alpha(d\theta), \hat{v}(t) \right\rangle_U dt\right]\\
    &+ \mathbb{E}\left[\int_0^T \left\langle E_t\Big[ \int_{0}^{+\infty} \phi(t,t+\theta)Q_{l}(t+\theta)\alpha(d\theta), \hat{v}(t) \right\rangle_U dt\right] \geqslant 0
\end{align*}
that is, recalling the notion of Hamiltonian function,
\[
\mathbb{E}\left[\int_0^T \left\langle E_t\Big[ \int_{0}^{+\infty} \phi(t,t+\theta)H_v^u(t+\theta)\alpha(d\theta), \hat{v}(t) \right\rangle_U dt\right] \geqslant 0.
\]
Referring to the proof of Theorem 1.5 in Cadenillas and Karatzas \cite{theorem15}, for all $v\in U$, we obtain that (\ref{condition neccesary SMP}) holds a.e..	
		\end{proof}

		\begin{remark}
			Condition (\ref{condition neccesary SMP}) is the SMP in integral form, which is equivalent to the following SMP without state constraint:
			\begin{equation}\label{SMP without state constrain}
				\begin{aligned}
					&\left\langle E_t\left[ \int_{0}^{+\infty}H^u_{v}(t+\theta)\phi(t,t+\theta)\alpha(d\theta)\right], u(t)\right\rangle_U   \\
					&	= \min_{v\in U} \left\langle E_t\left[ \int_{0}^{+\infty}H^u_{v}(t+\theta)\phi(t,t+\theta)\alpha(d\theta)\right],v\right\rangle_U  , \quad \text { a.e., a.s. }
				\end{aligned}
			\end{equation}
		\end{remark}

\subsection{Sufficient Conditions.}\label{chapter sufficient SMP}
		In this section, we will show that the necessary SMP with some convexity conditions can constitute the sufficient SMP for the optimal control problem (\ref{ISDDE with control})-(\ref{cost functional}). The convexity conditions are as follows:

\begin{assumption}\label{ass: convexity condition for sufficient SMP}
  For all $t\in[0,T]$ and any given $(p(t), q(t))$, $H(t,\cdot,\cdot,p(t), q(t))$ and $h(\cdot)$ are convex with respect to the corresponding variables, i.e. for each pair of $(x,v), (x',v')\in C_\lambda((-\infty,T];H)\times \mathcal{U}$, 
  \begin{align*}
        &H(t, x, v, p(t),q(t)) - H(t, x', v', p(t),q(t)) \\
        &\quad\geqslant \partial_x H(t, x',v',p(t),q(t))(x-x') + H_v(t, x',v',p(t),q(t))(v-v'),
  \end{align*}
  and for each $x_T, x_T'\in H$,
  \[
  h(x_T) - h(x_T') \geqslant h_{x_T}(x_T')(x_T - x_T').
  \]
\end{assumption}

		\begin{theorem}\label{sufficient SMP}
			(Sufficient conditions of optimality) Let $u(\cdot)\in\mathcal{U}_{ad}$ be an admissible control, $X^u(\cdot)$ be the corresponding trajectory and $(p(t),q(t))$ be the solution of adjoint equation (\ref{adjoint equation}). Suppose Assumptions~\ref{ass: Integrability}-\ref{ass: differentiability} and Assumption~\ref{ass: convexity condition for sufficient SMP} hold. If condition (\ref{condition neccesary SMP}) (or (\ref{SMP without state constrain})) holds for $u(\cdot)$, then $u(\cdot)$ is an optimal control of the stochastic control system with delay (\ref{ISDDE with control})-(\ref{cost functional}).
		\end{theorem}
		\begin{proof} 
        For any admissible control $v(\cdot)\in \mathcal{U}_{ad}$, let $X^v(\cdot)$ be the corresponding trajectory.
        We need to prove that $J\left(v(\cdot)\right) - J\left(u(\cdot)\right) \geqslant 0$. Note that
		\begin{align*}
		    J&\left(v(\cdot)\right) - J\left(u(\cdot)\right)\\
            &= \mathbb{E}\left[\int_{0}^{T}\left\{l\left(t,X^v_t,v_d(t)\right) - l\left(t,X^u(t),u_d(t)\right)\right\}dt\right] + \mathbb{E}\Big[h\left(X^v(T)\right) - h\left(X^u(T)\right)\Big]\\
            &=\mathbb{E}\left[\int_0^T \left(H (t, X^{v}_{t},v_{d}(t), p(t), q(t)) - H^u(t)\right) dt\right]  + \mathbb{E}\Big[h\left(X^v(T)\right) - h\left(X^u(T)\right)\Big]\\
            &\quad- \mathbb{E}\left[\int_0^T\left\langle b\left(t,X^v_t,v_d(t)\right) - b\left(t,X^u_t,u_d(t)\right), p(t)\right\rangle_H dt\right] \\
            &\quad-\mathbb{E}\left[\int_0^T\left\langle \sigma\left(t,X^v_t,v_d(t)\right) - \sigma\left(t,X^u_t,u_d(t)\right), q(t)\right\rangle_2 dt \right]\\
            &:= \nabla_1 + \nabla_2 - \nabla_3 - \nabla_4.
		\end{align*}
        By the convexity of Hamiltonian function $H$, it follow that
        \begin{align*}
            \nabla_1 
            &\geqslant \mathbb{E}\left[\int_0^T \partial_x^uH(t)(X^v_t - X^u_t)dt\right] + \mathbb{E}\left[\int_0^T \left\langle H^u_v(t), v_d(t) - u_d(t)\right\rangle_U dt\right].
        \end{align*}
        Since $h$ is convex, applying It\^o's formula (\ref{ito on 2 product}) to $\left\langle p(t), X^v(t) - X^u(t)\right\rangle_H$ on $[0,T]$ and taking expectation, we obtain 
            \begin{align*}
            \nabla_2 &\geqslant \mathbb{E}\left[\left\langle h_{x_{T}}(X^u(T)), X^v(T) -X^u(T) \right\rangle_H \right]  \\
            &= \mathbb{E}\left[\int_0^T \left\langle b(t, X^v_t, v_d(t)) - b(t, X^u_t, u_d(t)), p(t) 
                \right\rangle_H dt\right]\\
            &\quad + \mathbb{E}\left[\int_0^T \left\langle \sigma(t, X^v_t, v_d(t)) - \sigma(t, X^u_t, u_d(t)), q(t)
                \right\rangle_2 dt\right]\\
            &\quad - \mathbb{E}\left[\int_0^T \left\langle E_t\Big[\partial_x^{*,u}b(t)(p_{t+}) + \partial_x^{*,u}\sigma(t)(q_{t+}) + \partial_x^{*,u}l(t)(1)\Big] , X^v(t) -X^u(t)
                \right\rangle_H dt\right] \\
            & = \nabla_3 + \nabla_4 - \mathbb{E}\left[\int_0^T \left\langle \partial_x^ub(t)(X^v_t - X^u_t), p(t)\right\rangle_H dt\right]\\
            & \quad- \mathbb{E}\left[\int_0^T \left\langle \partial_x^u\sigma(t)(X^v_t - X^u_t), q(t)\right\rangle_2 dt\right] -\mathbb{E}\left[\int_0^T \partial_x^ul(t)(X^v_t - X^u_t) dt\right]\\
            &= \nabla_3 + \nabla_4 - \mathbb{E}\left[\int_0^T \partial_x^uH(t)(X^v_t - X^u_t)dt\right],
            \end{align*}
            where the second last step follows from the duality relationship (\ref{eq:adapted-duality}) and the last step holds by the definition of Hamiltonian function $H$.
			Therefore, by Corollary~\ref{cor:section3-specialization} and maximum condition (\ref{condition neccesary SMP}) or (\ref{SMP without state constrain}), we obtain
			\begin{align*}
			    	J\left(v(\cdot)\right) - J\left(u(\cdot)\right) &\geqslant \mathbb{E}\left[\int_0^T \left\langle H^u_v(t), v_d(t) - u_d(t)\right\rangle_U dt\right]\\
                    &= \mathbb{E}\left[\int_{0}^{T}\left\langle E_t\left[ \int_{0}^{+\infty}\phi(t,t+\theta)H^u_{v}(t+\theta)\alpha(d\theta)\right],v(t)-u(t)\right\rangle_U dt\right] \geqslant 0.
			\end{align*}
			Since $v(\cdot)\in \mathcal{U}_{ad}$ is arbitrary, $u(\cdot)$ is an optimal control.
		\end{proof}
		

\section{Application.} \label{chapter applications}

    In this section, we solve a linear-quadratic (LQ) control problem of infinitely delayed SEE. Applying the theoretical result developed in Section~\ref{section SMP}, we aim to derive the optimal control explicitly. 

Consider the following linear stochastic control system with delay:
\begin{equation}\label{eq: example state dynamics}
    \left\{
			\begin{aligned}
				dX^v(t) &= \Big[A(t) X^v_t + B(t)v_d(t)\Big]dt + \Big[C(t) X^v_t + D(t)v_d(t)\Big] dW(t), \quad t\in[0,T],\\
				X^v(t) &= \gamma(t), \quad v(t) = \varphi(t), \qquad t\in (-\infty, 0],
			\end{aligned}
			\right.
\end{equation}

Denote by $\mathcal{U}_{ad}$ the set of all admissible controls $v(\cdot)$ of the form:
		\begin{equation*}
		    		v(t)= \left\{\begin{aligned}
				&\varphi(t), &&t \in(-\infty, 0];\\ 
				&v(t) \in \mathcal{M}_{\mathcal{F}}^{2}\left(0, T ; U\right), \quad v(t) \in U, \text { a.s. },   &&t \in[0,T].\end{aligned}\right.
		\end{equation*}
        
Our goal is to minimize the following quadratic cost functional 
		\begin{equation}\label{example cost functional}
			J(v(\cdot))= \frac{1}{2}\mathbb E\left[\int_0^T\left\{ \left\langle L(t)X^v_t, X^v_t\right\rangle_H + \left\langle\tilde{L}(t)v_d(t), v_d(t)\right\rangle_U \right\} dt + \left\langle X^v(T), X^v(T)\right\rangle_H\right].
		\end{equation}
subject to (\ref{eq: example state dynamics}) over $v(\cdot)\in \mathcal{U}_{ad}$.

In the above, the initial datum $\gamma(\cdot)\in \mathcal{M}_\mathcal{F}^2(-\infty, 0 ; H)$ satisfies $\gamma_0 \in L^2(\mathcal{F}_0;C_\lambda((-\infty,0];H))$ and $\varphi(\cdot)\in \mathcal{M}_\mathcal{F}^2(-\infty, 0 ; U)$. 
The $\mathcal{F}$-adapted functions $A(\cdot), B(\cdot), C(\cdot), D(\cdot),L(\cdot), \tilde{L}(\cdot)$ are uniformly bounded, and $\tilde{L}^{-1}(\cdot)$ is also bounded.  
Moreover, to ensure the convexity of the Hamiltonian function, we assume that $L(\cdot), \tilde{L}(\cdot)$ are symmetric and nonnegative definite a.e., a.s. and $\tilde{L}(\cdot)$ is uniformly positive definite 	a.e., a.s. so that $\tilde{L}(\cdot)$ is invertible.

The Hamiltonian function $H$ takes the form
		\begin{align*}
			H&\left(t, x, v,p,q\right)= \left\langle A(t)x + B(t)v, p  \right\rangle_H + \left\langle C(t)x + D(t)v, q  \right\rangle_2 + 
            \frac{1}{2}\left\langle L(t)x, x\right\rangle_H +  \frac{1}{2}\langle\tilde{L}(t)v, v\rangle_U.
		\end{align*}
        
Let $u(\cdot)$ be an optimal control and $X^u(\cdot)$ be the corresponding optimal trajectory to equation (\ref{eq: example state dynamics}).
The associated adjoint equation becomes
		\begin{equation}\label{application 2 adjoint equation}
			\left\{\begin{aligned}
				-d p(t)= & \bigg\{ E_t\Big[\partial^{*,u}_xb(t)(p_{t+}) + \partial^{*,u}_x\sigma(t)(q_{t+}) + \partial_x^{*,u}l(t)(1)\Big] 
				 \bigg\} d t
				- q(t) dW(t), \quad t\in[0,T]; \\
				p(T)= & X^u(T), \qquad p(t)=0, \quad t\in(T,+\infty), \qquad q(t) = 0, \quad  t\in[T,+\infty),
			\end{aligned}\right.
		\end{equation}

In this case, the non-anticipative derivatives are denoted as follows: for each $Z\in \mathcal{M}^2_{\mathcal{F}}(-\infty,T;H)$, 
\begin{align*}
	\partial_{x}^ub(Z)=A(t)Z_t,\quad \partial_{x}^u \sigma(Z)=C(t)Z_t, \quad \partial_{x}^u l(Z)=\langle L(t) X^u(t),Z(t)\rangle_H.
\end{align*}

        \begin{remark}
            By Theorem~\ref{thm:ISFDE-wellposed}, the state equation (\ref{eq: example state dynamics}) has a unique solution $X^v(\cdot) \in \mathcal{M}_\mathcal{F}^2\left(-\infty, T; H\right)$ for all $v(\cdot)\in \mathcal{U}_{ad}$. The well-posedness of the adjoint equation (\ref{application 2 adjoint equation}) holds by Theorem~\ref{thm:IABSDE-wellposed}.
        \end{remark}
        
The necessary maximum condition (\ref{condition neccesary SMP}) or (\ref{necessary SMP}) implies that 
\[
E_t\left[ \int_0^{+\infty}\phi(t, t+\theta)B^{*}(t+\theta)p(t+\theta)\alpha(d\theta) + \int_0^{+\infty}\phi(t, t+\theta)D^{*}(t+\theta)q(t+\theta)\alpha(d\theta) \right] + \tilde{L}(t)u(t) = 0.
 \]

By the convexity of $H$, we know that the above formula is sufficient. Thus, the optimal control of the LQ optimization problem (\ref{eq: example state dynamics})-(\ref{example cost functional}) is given by 
\[
u(t) = -\tilde{L}^{-1}(t)\cdot E_t\left[ \int_0^{+\infty}\phi(t, t+\theta)\Big\{B^{*}(t+\theta)p(t+\theta) + D^{*}(t+\theta)q(t+\theta)\Big\}\alpha(d\theta) \right],
\]
where $(p(\cdot), q(\cdot))$ solve the equation (\ref{application 2 adjoint equation}).

		\section{Conclusion.}
        This paper is concerned with optimal control problems for SEE with infinite delay in Hilbert spaces. Our control system allows general functional dependence on the whole past trajectory of the state dynamics and includes delayed controls via an integral with respect to a general finite measure. Within a fading-memory framework, we establish SMP and derive both necessary and sufficient optimality conditions. In this way, we extend the SMP from finite-delay systems to the infinite-delay setting, and generalize existing infinite-delay results beyond the integral-delay case. As an application, we study an infinite-delay LQ control problem and obtain the explicit optimal control.

        Several directions may be considered in future research. A natural extension is to introduce an unbounded infinite-dimensional operator into the state equation, as in \cite{SDDE-12.delay-2type-2, SPDE-12.delay-2type-1, songjian2025}, which would connect the present framework with stochastic partial differential equation (SPDE) with delay. 
        Another possible direction is to consider problems with partial information or more general control constraints.
		It would also be meaningful to develop numerical methods for the approximation of the state equation, the adjoint equation, and the optimal controls.
		\\
		\\
		\textbf{Acknowledgements.} The author thanks Professor Shuzhen Yang and Dr. Yiming Chen for valuable comments and suggestions, which have helped improve the presentation of this paper.

		\bibliography{citation}
		
	\end{document}